\documentclass{article}
\usepackage{fullpage}
\usepackage[utf8]{inputenc}
\usepackage{amsmath,amsthm,amssymb}
\usepackage{cite}
\usepackage{url}
\usepackage[nobysame]{amsrefs}
\usepackage{enumerate}
\usepackage{caption}
\usepackage{float}
\usepackage{tikz}
\usetikzlibrary{calc}
\usetikzlibrary{decorations.shapes, shapes.geometric}
\usepackage{colonequals}

\usepackage{titling}
\thanksmarkseries{arabic}

\usepackage{algorithm}
\usepackage[noend]{algpseudocode}

\theoremstyle{definition}
\newtheorem{theorem}{Theorem}
\newtheorem{definition}[theorem]{Definition}

\newtheorem{corollary}[theorem]{Corollary}
\newtheorem{lemma}[theorem]{Lemma}
\newtheorem{remark}[theorem]{Remark}

\newtheorem{notation}[theorem]{Notation}

\allowdisplaybreaks

\usepackage[unicode,colorlinks=true,citecolor=green!40!black,linkcolor=red!20!black,urlcolor=blue!40!black,filecolor=cyan!30!black]{hyperref}

\title{Color-avoiding connected spanning subgraphs with minimum number of edges}

\author{J\'{o}zsef Pint\'{e}r\thanks{Department of Stochastics, Institute of Mathematics, Budapest University of Technology and Economics.} \textsuperscript{,}\thanks{HUN-REN--BME Stochastics Research Group.} \and Kitti Varga\thanks{Department of Computer Science and Information Theory, Faculty of Electrical Engineering and Informatics, Budapest University of Technology and Economics.} \textsuperscript{,}\thanks{HUN-REN--ELTE Egerv\'{a}ry Research Group.} \textsuperscript{,}\thanks{MTA-ELTE Momentum Matroid Optimization Research Group.}}

\begin{document}

\maketitle

\begin{abstract}
 We call a (not necessarily properly) edge-colored graph edge-color-avoiding connected if after the removal of edges of any single color, the graph remains connected. For vertex-colored graphs, similar definitions of color-avoiding connectivity can be given.
 
 In this article, we investigate the problem of determining the maximum number of edges that can be removed from either an edge- or a vertex-colored, color-avoiding connected graph so that it remains color-avoiding connected. First, we prove that this problem is NP-hard, and then, we give a polynomial-time approximation algorithm for it. To analyze the approximation factor of this algorithm, we determine the minimum number of edges of color-avoiding connected graphs on a given number of vertices and with a given number of colors. Furthermore, we also consider a generalization of edge-color-avoiding connectivity to matroids.
\end{abstract}

\begin{quote}
\end{quote}
\vspace{5mm}

\section{Introduction}

The robustness of networks against random errors and targeted attacks has attracted a great deal of research interest \cites{previouswork1, previouswork2, previouswork3, previouswork4}. The robustness of a network refers to its capacity to maintain some degree of connectivity after the removal of some edges or vertices of the network. 

Although the standard frameworks of error or attack tolerance in complex networks can be really useful in industrial practices, we can develop a more efficient framework if we take into account that some parts of the network might share some vulnerabilities. A characteristic example is the case of public transport networks, where the edges of the underlying graph are colored according to the mode of transportation such as rail, road, ship or air transport. Experience shows that excessive snowing usually has a greater impact on the railway than on underground transportation. In extreme cases, these weather conditions might even paralyze the whole railway traffic, which we can think of (from a network theoretical point of view) that all the edges corresponding to railway transportation disappear from the network. Thus it is useful to know which vertices in the network are available from each other without using any edges corresponding to the railway transportation. In this manner, we can consider the network reliable if even after the elimination of any single mode of transportation, the whole or a significant part of the network remains connected. Another example is the case of communication networks where the vertices represent routers, which are colored according to which country the corresponding router is registered to. If, for safety reasons, we want to ensure that no country can intercept our message, then we need multiple paths in the network between the sender and the receiver such that each country is avoided in at least one of these paths, and send our message divided into many parts through these paths.

\paragraph{Previous work.} The above mentioned concepts were introduced as color-avoiding connectivity, first for vertex-colored graphs by Krause et al.~\cite{Krause1} in 2016 with the motivation to develop a framework which can treat the heterogeneity of multiple vulnerable classes of vertices, and they demonstrated how this can be exploited to maintain functionality of a complex network by utilizing multiple paths, mostly on communication networks. Krause et al.\ extended this original theory in~\cite{Krause2}. They analyzed how the color frequencies affect the robustness of the networks. For unequal color frequencies, they found that the colors with the largest frequencies control vastly the robustness of the network, and colors of small frequency only play a little role. In~\cite{Krause3}, color-avoiding connectivity was further extended from vertex-colored graphs to edge-colored ones, and similar problems as those for vertex-colored graphs were also considered.

Ráth et al.~\cite{rath} investigated the color-avoiding bond percolation of edge-colored Erd\H{o}s--R\'{e}nyi random graphs. They analyzed the fraction of vertices contained in the giant edge-color-avoiding connected component and proved that its limit can be expressed in terms of probabilities associated to edge-colored branching process trees. The work~\cite{lichevschapira} of Lichev and Schapira includes some simplification and generalization of these results as well as some finer results on the size of the largest edge-color-avoiding connected component. Lichev also described the phase transition of the largest edge-color-avoiding connected component between the supercritical and the intermediate regime~\cite{lichev}.

Molontay and Varga~\cite{VargaMolontay} determined the computational complexity of finding the color-avoiding connected components of a graph. They also generalized the concept of color-avoiding connectivity by making the vertices or edges more vulnerable by assigning a list of colors to them. They proved that for certain concepts, the color-avoiding connected components can be found in polynomial time, while for others this problem is NP-hard.

Edge-color-avoiding colorings were also studied under the name courteous edge-colorings by DeVos et al.~\cite{seymour} in 2006. Graphs with 1-courteous edge-colorings are exactly the edge-color-avoiding connected graphs. In that article, they gave interesting upper bounds on the number of colors needed to courteously color an arbitrary graph.

\paragraph{Our results.} When operating a network, we might want to reduce the maintenance cost while retaining some desired properties of the network. In this work, we investigate the problem of finding color-avoiding connected spanning subgraphs with minimum number of edges in edge- or vertex-colored graphs. We also consider a generalization of edge-color-avoiding connectivity to matroids and investigate a similar problem. First, we prove that all these problems are NP-hard, then we present polynomial-time approximation algorithms for them. In this article, we extend the results of the master's thesis~\cite{diplomamunka} to matroids, and also improve and more accurately analyze our algorithms.

\section{Color-avoiding connected graphs and courteously colored matroids}

In this article, we study color-avoiding connected graphs and courteously colored matroids. All graphs considered in this article are not necessarily properly colored\footnote{A proper edge-coloring is an assignment of colors to the edges in which no two incident edges receive the same color. A proper vertex-coloring is an assignment of colors to the vertices in which no two adjacent vertices receive the same color.}. First, we recall some important definitions and notation.

The set of positive integers is denoted by $\mathbb{Z}_+$. For two sets $X$ and $Y$, the \emph{set difference} of $X$ and $Y$ is denoted by $X-Y$. Given a graph $G = (V, E)$ and a subset of edges $E' \subseteq E$, let $G - E'$ denote the graph that is obtained from $G$ by deleting the edges of $E'$ from it. If $E' = \{ e \}$ for some edge $e$ of $G$, then $G - \{ e \}$ is abbreviated by $G - e$.
A graph is called \emph{$k$-edge-connected} if it remains connected whenever fewer than $k$ edges are removed. Similarly, a graph is called \emph{$k$-vertex-connected} if it has more than $k$ vertices and remains connected whenever fewer than $k$ vertices are removed.

A \emph{matroid} $\mathcal{M} = (S, \mathcal{I})$ is a pair formed by a finite (possibly empty) \emph{ground set} $S$ and a family of subsets $\mathcal{I} \subseteq 2^S$ called \emph{independent sets} satisfying the \emph{independence axioms}:
\begin{enumerate}
 \item[(I1)] $\emptyset \in \mathcal{I}$,
 \item[(I2)] for any $X,Y \subseteq S$ with $X \subseteq Y$, if $Y \in \mathcal{I}$, then $X \in \mathcal{I}$,
 \item[(I3)] for any $X,Y \in \mathcal{I}$ with $|X| < |Y|$, there exists $e \in Y-X$ such that $X \cup \{e\} \in \mathcal{I}$.
\end{enumerate}

The maximal independent subsets of $S$ are called \emph{bases}. The \emph{rank} of a set $X \subseteq S$ in the matroid, denoted by $r(X)$, is the maximum size of an independent subset of $X$. The \emph{rank of the matroid} is the rank of its ground set. If $\mathcal{M}$ is a matroid on the ground set $S$ and $T \subseteq S$, then the \emph{restriction} of $\mathcal{M}$ to $T$, or in other words, the \emph{deletion} of $S-T$ from $\mathcal{M}$, is the matroid $\mathcal{M}\big|_T \colonequals (T, \mathcal{I'})$, or also denoted by $\mathcal{M} \setminus (S - T) \colonequals (T, \mathcal{I'})$, where $\mathcal{I}' \colonequals \{ X \subseteq T \mid X \in \mathcal{I} \}$. Furthermore, if the rank of $\mathcal{M}\big|_T$ equals that of $\mathcal{M}$, i.e.\ $r(T) = r(S)$, then we say that $\mathcal{M}\big|_T$ is a \emph{rank-preserving restriction} of $\mathcal{M}$. A \emph{coloring} of a matroid is an arbitrary assignment of colors to the elements of its ground set. A \emph{graphic matroid} is a matroid whose independent sets can be represented as the edge sets of forests of a graph. If the underlying graph is connected, then the bases of the graphic matroid are the spanning trees of the graph; for an example, see Figure~\ref{fig:graphic_matroid}.

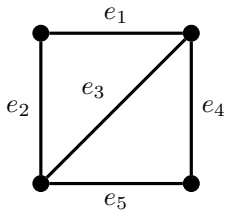
\begin{figure}[H]
\centering
\begin{tikzpicture}[scale=2]
 \tikzstyle{vertex}=[draw,circle,fill,minimum size=6,inner sep=0]
 
 \node[vertex] (a1) at (0,0)  {};
 \node[vertex] (b1) at (0,1) {};
 \node[vertex] (c1) at (1,1) {};
 \node[vertex] (d1) at (1,0) {};
 
 \draw[very thick] (a1) -- (b1) node[midway, left] {$e_2$};
 \draw[very thick] (c1) -- (d1) node[midway, right] {$e_4$};
 \draw[very thick] (b1) -- (c1) node[midway, above] {$e_1$};
 \draw[very thick] (a1) -- (d1) node[midway, below] {$e_5$};
 \draw[very thick] (a1) -- (c1) node[midway, above left] {$e_3$};
\end{tikzpicture}
\caption{An example for a graphic matroid: the ground set of the matroid is $\{ e_1, e_2, e_3, e_4, e_5 \}$ and the non-independent sets are $\{ e_1, e_2, e_3 \}$, $\{ e_3, e_4, e_5 \}$, $\{ e_1, e_2, e_3, e_4 \}$, $\{ e_1, e_2, e_3, e_5 \}$, $\{ e_1, e_2, e_4, e_5 \}$, $\{ e_1, e_3, e_4, e_5 \}$, $\{ e_2, e_3, e_4, e_5 \}$, and $\{ e_1, e_2, e_3, e_4, e_5 \}$.}
\label{fig:graphic_matroid}
\end{figure}

Since the number of independent sets can be exponential in the size of the ground set, for matroid algorithms it is usually assumed that the matroid is accessed through an \emph{oracle}. An oracle can be thought of as a black box that, given an input set $X \subseteq S$, outputs certain properties of the set, e.g.\ whether $X$ is independent or not (\emph{independence oracle}), or the rank of $X$ (\emph{rank oracle}). Then the complexity of a matroid algorithm is measured by the number of oracle calls and other conventional elementary steps. There are various types of oracles having the same computational power in the sense that any of them can be simulated by using a polynomial number of calls to any of the others, measured in terms of the size of the ground set. For further details, we refer the interested reader to~\cites{robinson1980computational,hausmann1981algorithmic,coullard1996independence}. Here we assume that a matroid is accessed through an independence oracle that can determine whether a subset of the ground set is independent or not.

Now we present the definitions of color-avoiding connectivity in edge- and vertex-colored graphs and a generalization of edge-color-avoiding connectivity to matroids.

\begin{definition}
 We say that an edge-colored graph $G$ is \emph{edge-color-avoiding connected} if after the removal of the edges of any single color from $G$, the remaining graph is connected.
\end{definition}

For two small examples on the definition of edge-color-avoiding connectivity, see Figure~\ref{fig:edgeexample}.

\begin{figure}[H]
\centering
\begin{tikzpicture}[scale=2]
 \tikzstyle{vertex}=[draw,circle,fill,minimum size=10,inner sep=0]
 \tikzset{paint/.style={draw=#1!50!black, fill=#1!50}, decorate with/.style = {decorate, decoration={shape backgrounds, shape=#1, shape size = 5pt, shape sep = 6pt}}}
 \tikzstyle{edge_red}=[draw, decorate with = isosceles triangle, paint = red]
 \tikzstyle{edge_blue}=[draw, decorate with = rectangle, decoration = {shape size = 4pt, shape sep = 6.5pt}, paint = blue]
 \tikzstyle{edge_green}=[draw, decorate with = diamond, paint = green]
 \tikzstyle{edge_yellow}=[draw, decorate with = star, paint = yellow]
 
 \node[vertex] (a1) at (0,0)  {};
 \node[vertex] (b1) at (0,1) {};
 \node[vertex] (c1) at (1,1) {};
 \node[vertex] (d1) at (1,0) {};
 
 \draw[edge_red] (a1) -- (b1);
 \draw[edge_blue] (c1) -- (d1);
 \draw[edge_green] (b1) -- (c1);
 \draw[edge_yellow] (a1) -- (d1);
  
 \begin{scope}[shift={(3,0)}]
 \node[vertex] (a1) at (0,0)  {};
 \node[vertex] (b1) at (0,1) {};
 \node[vertex] (c1) at (1,1) {};
 \node[vertex] (d1) at (1,0) {};
 
 \draw[edge_red] (a1) -- (b1);
 \draw[edge_blue] (c1) -- (d1);
 \draw[edge_green] (b1) -- (c1);
 \draw[edge_blue] (d1) -- (a1);
 \end{scope}
\end{tikzpicture}
\caption{An example for an edge-color-avoiding connected graph (left) --- after the removal of edges of any single color, there remains a Hamiltonian path ---, and an example for a not edge-color-avoiding connected graph (right) --- after the removal of the blue (denoted by squares) edges, the bottom right vertex becomes isolated.}
\label{fig:edgeexample}
\end{figure}
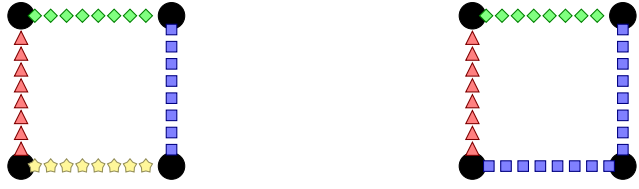

The following lemma directly follows from the definition.

\begin{lemma} \label{lem:edge}
 Let $G$ be an edge-colored graph in which every edge has a different color. Then $G$ is edge-color-avoiding connected if and only if it is $2$-edge-connected.
\end{lemma}

Clearly, an edge-colored graph is edge-color-avoiding connected if and only if after the removal of the edges of any single color, there exists a spanning tree in the remaining graph. This motivates the introduction of the following definition for matroids, which we call, after DeVos et al.~\cite{seymour}, courteously colored matroids.

\begin{definition} \label{def:courteous_matroid}
 Let $\mathcal{M}$ be a matroid, whose ground set is colored. We say that $\mathcal{M}$ is a \emph{courteously colored matroid} if after the deletion of the elements of any single color from $\mathcal{M}$, the rank of the matroid does not change.
\end{definition}

Thus a matroid is courteously colored if and only if after the deletion of the elements of any single color from the ground set, at least one basis remains intact. In particular, a graphic matroid is courteously colored if and only if each component of the corresponding graph is edge-color-avoiding connected. Another simple example is the case of \emph{uniform matroids}. The ground set of a uniform matroid $U_{n,k}$ is of size $n$, and its independent sets are those subsets of the ground set whose cardinality is at most $k$ for some integer $0 \le k \le n$. We note that $U_{n,k}$ is graphic if and only if $k \in \{ 0, 1, n-1, n \}$, see~\cite{book:recski} for further details. It is not difficult to see that the uniform matroid $U_{n,k}$ is courteously colored if and only if there exists no color which is assigned to at least $n-k+1$ elements: since the bases of $U_{n,k}$ are exactly the subsets of size $k$, after the deletion of the elements of any single color at least one basis remains intact if and only if at least $k$ elements have different colors from the deleted one.

Now we present two definitions of color-avoiding connectivity for vertex-colored graphs describing slightly different phenomena.

\begin{definition}
 We say that two vertices $u$ and $v$ of a vertex-colored graph $G$ are \emph{vertex-$c$-avoiding connected} for some color $c$ if there exists a $u$-$v$ path, and either at least one of $u$ and $v$ is of color $c$, or at least one $u$-$v$ path does not contain a vertex of color~$c$. If any two vertices of $G$ are vertex-$c$-avoiding connected for any color $c$, then we say that $G$ is \emph{vertex-color-avoiding connected}.
\end{definition}

\begin{definition}
 We say that two vertices $u$ and $v$ of a vertex-colored graph $G$ are \emph{internally vertex-$c$-avoiding connected} for some color $c$ if there exists a $u$-$v$ path containing no internal vertices of color $c$. If any two vertices of $G$ are internally vertex-$c$-avoiding connected for any color $c$, then we say that $G$ is \emph{internally vertex-color-avoiding connected}.
\end{definition}

Note that if a graph is internally vertex-color-avoiding connected, then it is vertex-color-avoiding connected as well, but not every vertex-color-avoiding connected graph is internally vertex-color-avoiding connected. For some small examples on the definitions of vertex- and internally vertex-color-avoiding connectivity, see Figure~\ref{fig:vertexexample}. 

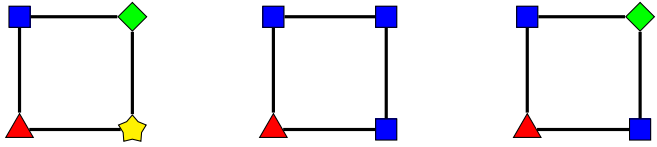
\begin{figure}[H]
\centering
\begin{tikzpicture}[scale=1.5]
 \tikzstyle{vertex_red}=[draw, shape=regular polygon, regular polygon sides=3, fill=red, minimum size=12pt, inner sep=0]
 \tikzstyle{vertex_blue}=[draw, shape=rectangle, fill=blue, minimum size=8pt, inner sep=0]
 \tikzstyle{vertex_green}=[draw, shape=diamond, fill=green, minimum size=11pt, inner sep=0]
 \tikzstyle{vertex_yellow}=[draw, shape=star, fill=yellow, minimum size=11pt, inner sep=0]
 
 \node[vertex_red] (a1) at (0,0) {};
 \node[vertex_blue] (a2) at (0,1) {};
 \node[vertex_yellow] (a3) at (1,0) {};
 \node[vertex_green] (a4) at (1,1) {};

 \draw[very thick] (a1) -- (a2) -- (a4) -- (a3) -- (a1);
 
 \begin{scope}[shift={(2.25,0)}]
 \node[vertex_red] (a5) at (0,0) {};
 \node[vertex_blue] (a6) at (0,1) {};
 \node[vertex_blue] (a7) at (1,0) {};
 \node[vertex_blue] (a8) at (1,1) {};

 \draw[very thick] (a5) -- (a6) -- (a8) -- (a7) -- (a5);
 \end{scope}
 
 \begin{scope}[shift={(4.5,0)}]
 \node[vertex_red] (a9) at (0,0) {};
 \node[vertex_blue] (a10) at (0,1) {};
 \node[vertex_blue] (a11) at (1,0) {};
 \node[vertex_green] (a12) at (1,1) {};

 \draw[very thick] (a9) -- (a10) -- (a12) -- (a11) -- (a9);
 \end{scope}
\end{tikzpicture}
\caption{An example for a vertex- and internally vertex-color-avoiding connected graph (left) --- after the removal of vertices of any single color, there remains a Hamiltonian path, thus for any two vertices, there exists a path containing no internal vertices of the removed color ---, and an example for a vertex- but not internally vertex-color-avoiding connected graph (middle) --- there exists no path between the bottom left and the top right vertices which avoids internal vertices of color blue (denoted by squares), thus these two vertices are not internally vertex-blue-avoiding connected, however, they are vertex-blue-avoiding connected since the color of the top right vertex is blue. Finally, an example of a graph which is neither vertex- nor internally vertex-color-avoiding connected (right) --- the red (denoted by triangle) and green (denoted by rhombus) vertices are neither vertex- nor internally vertex-blue-avoiding connected.}
\label{fig:vertexexample}
\end{figure} 

Although vertex-color-avoiding connectivity does not imply internally vertex-color-avoiding connectivity in general, the following lemma shows that this implication holds for those graphs in which every vertex has a different color. The proof of this lemma directly follows from the definitions.

\begin{lemma} \label{lem:vertex}
 Let $G$ be a vertex-colored graph in which every vertex has a different color. Then the following are equivalent.
 
 \begin{enumerate}
  \item The graph $G$ is vertex-color-avoiding connected.
  \item The graph $G$ is internally vertex-color-avoiding connected.
  \item The graph $G$ is 2-vertex-connected.
 \end{enumerate}
\end{lemma}

For convenience, let us introduce the following notation.

\begin{notation} \label{not:comp}
 Given an edge- or a vertex-colored graph $G$ and a color $c$, we denote by $G_{\overline{c}}$ the graph which can be obtained from $G$ by removing the edges or the vertices of color $c$ from it.
 
 Given a matroid $\mathcal{M}$ whose ground set is colored and given a color $c$, we denote by $\mathcal{M}_{\overline{c}}$ the matroid which can be obtained from $\mathcal{M}$ by deleting the elements of color $c$ from it.
\end{notation}

\section{Main results}

 In this section, we study the problems of finding edge-, vertex- and internally vertex-color-avoiding connected spanning subgraphs with minimum number of edges in edge- or vertex-colored graphs, and also finding courteously colored rank-preserving restrictions to a set of minimum size in courteously colored matroids.

\subsection{Complexity}

 It is not difficult to see that an edge- or vertex-colored graph $G$ has a color-avoiding connected spanning subgraph if and only if $G$ is color-avoiding connected, and similarly, a colored matroid $\mathcal{M}$ has a courteously colored rank-preserving restriction if and only if $\mathcal{M}$ is courteously colored.

\begin{theorem} \label{thm:complexity}
 Given a matroid $\mathcal{M} = (S, \mathcal{I})$ whose ground set is colored and given a positive integer $m$, it is NP-complete to decide whether $\mathcal{M}$ has a courteously colored rank-preserving restriction to a subset $T \subseteq S$ of size at most $m$. Furthermore, this problem remains NP-complete even for graphic matroids.
\end{theorem}
\begin{proof}
 The problem is clearly in NP. Now we show that the problem is NP-hard even for graphic matroids. As we observed before Definition~\ref{def:courteous_matroid}, a graphic matroid is courteously colored if and only if each component of the corresponding graph is edge-color-avoiding connected. Lemma~\ref{lem:edge} implies that for those connected graphs in which every edge has a different color and for the choice of $m = \big| V(G) \big|$, our problem is equivalent to deciding whether the graph contains a Hamiltonian cycle, which is known to be NP-complete~\cite{hamilton}.
\end{proof}

\begin{corollary} \label{cor:ECA-subgraph_NP-complete}
 Given an edge-colored graph $G = (V, E)$ and a positive integer $m$, it is NP-complete to decide whether $G$ has an edge-color-avoiding connected spanning subgraph with at most $m$ edges.
\end{corollary}

Now we prove an analogue of Corollary~\ref{cor:ECA-subgraph_NP-complete} for the concepts of vertex- and internally vertex-color-avoiding connectivity.

\begin{theorem}\label{thm:vertexNP}
 Given a vertex-colored graph $G$ and a positive integer $m$, it is NP-complete to decide whether $G$ has a vertex-color-avoiding connected spanning subgraph with at most $m$ edges. Similarly, it is also NP-complete to decide whether $G$ has an internally vertex-color-avoiding connected spanning subgraph with at most $m$ edges.
\end{theorem}
\begin{proof}
 Both problems are clearly in NP. Now we show that they are NP-hard. If every vertex has a different color, then by Lemma~\ref{lem:vertex}, for the choice of $m = \big| V(G) \big|$, both problems are equivalent to deciding whether the graph contains a Hamiltonian cycle, which is known to be NP-complete~\cite{hamilton} even for subcubic graphs~\cite{hamilton-subcubic}.
\end{proof}

\subsection{Approximation algorithm for courteously colored matroids}

As is clear from the proofs of Theorems~\ref{thm:complexity} and~\ref{thm:vertexNP}, in the case of those connected graphs whose edges or vertices are all of different colors, we want to find a 2-edge- or a 2-vertex-connected spanning subgraph with minimum number of edges, respectively, which are NP-hard problems. However, there exist polynomial-time approximation algorithms for these problems. Khuller and Vishkin~\cite{KhullerVishkin} provided a $3/2$-approximation algorithm for 2-edge-connectivity and a $5/3$-approximation algorithm for 2-vertex-connectivity: they modified the depth-first search algorithm so that it does not just find a spanning tree, but a minimally 2-edge- or 2-vertex-connected\footnote{A graph is called minimally $k$-edge- or $k$-vertex-connected if it is $k$-edge- or $k$-vertex-connected but after the removal of any of its edges, the obtained graph is not $k$-edge- or $k$-vertex-connected, respectively.} spanning subgraph. Gabow et al.~\cite{gabow} presented a $\big( 1+\frac{2}{k} \big)$-approximation algorithm for finding a $k$-edge-connected spanning subgraph with minimum number of edges with the use of linear programming. Currently, the best known approximation factor for finding a 2-edge-connected spanning subgraph with minimum number of edges is~$4/3$ by Hunkenschröder et al.~\cite{vempala}, and for finding 2-vertex-connected spanning subgraphs with minimum number of edges Cheriyan and Thurimella~\cite{cheriyan} gave a $3/2$-approximation algorithm, but there exist better performing algorithms if the input graph satisfies some additional conditions --- for example, see~\cites{krysta, narayan}.

In this section, we present a polynomial-time approximation algorithm for finding a courteously colored rank-preserving restriction of a matroid to a set of minimum size. But first, we give a lower bound on the number of elements of a courteously colored matroid. Later, this lower bound is used for determining the approximation factor of our algorithm.

\begin{theorem} \label{thm:edge}
 Let $\mathcal{M} = (S, \mathcal{I})$ be a courteously colored matroid colored with exactly $k \in \mathbb{Z}_+$ colors, and let $r=r(S)$. If $r=0$, then $|S| \ge 0$, and if $r \ge 1$, then $k \ge 2$ and $|S| \ge {\footnotesize \Big\lceil} \frac{k \cdot r}{k-1} {\footnotesize \Big\rceil}$, and these lower bounds are tight. 
\end{theorem}
\begin{proof}
 If $r=0$, then every matroid is courteously colored including the one with the empty ground set; thus, the lower bound is tight in this case.
 
 Suppose to the contrary that there exists a courteously colored matroid $\mathcal{M}$ colored with $k=1$ color~$c$ and with rank $r \ge 1$. Then by the definition of courteous colorings, $\mathcal{M}_{\overline{c}}$ has rank $r$. On the other hand, the ground set of $\mathcal{M}_{\overline{c}}$ is the empty set, which is a contradiction.
 
 \medskip
 
 Now consider the case $k \ge 2$. Let $\mathcal{M}$ be a courteously colored matroid with rank $r$, where the elements of the ground set are colored with exactly $k$ colors. Then for any of these $k$ colors $c$, the matroid $\mathcal{M}_{\overline{c}}$ has rank $r$, thus its ground set has at least $r$ elements. That sums up to at least $k\cdot r$ elements, where every element is counted exactly $k-1$ times, so the number of elements is at least ${\footnotesize \Big\lceil} \frac{k\cdot r}{k-1} {\footnotesize \Big\rceil}$.
 
 To show that these lower bounds are tight, we construct courteously colored graphic matroids of rank $r$ on ${\footnotesize \Big\lceil} \frac{k \cdot r}{k-1} {\footnotesize \Big\rceil}$ elements which are colored with exactly $k$ colors for any $k \ge 2$. More precisely, we construct an edge-color-avoiding connected graph $G=(V,E)$ on $r+1$ vertices and with $r + {\footnotesize \Big\lceil} \frac{r}{k-1} {\footnotesize \Big\rceil} = {\footnotesize \Big\lceil} \frac{k\cdot r}{k-1} {\footnotesize \Big\rceil}$ edges, where the edges are colored with exactly $k$ colors. Let
 \[ C \colonequals \{ 0, 1, \ldots, k-1 \} \]
 be the color set, let
 \[ V \colonequals \{ v_0, \ldots, v_r\} \]
 with $r \ge k-1$, and let
 \[ E_i \colonequals \big\{ v_j v_{j+1} \bigm| j \in \{ 0, 1, \ldots, r-1 \} \text{ and } j \equiv i \!\!\! \pmod{k-1} \big\} \]
 be the set of edges of color $i$ for any $i \in \{ 0, 1, \ldots, k-2 \}$, and let
 \[ E_{k-1} \colonequals \big\{ v_j v_{\max(j+k-1,r)} \bigm| j \in \{ 0, \ldots, r-1 \} \text{ and } j \equiv 0 \!\!\! \pmod{k-1} \big\} \]
 be the set of edges of color $k-1$, and let $E = \bigcup_{i \in C} E_i$. Note that $G$ might have a pair of parallel edges between the vertices $v_{r-1}$ and $v_r$; for an example see Figure~\ref{fig:edge-proof-gencon}.
 
 It is not difficult to show that $G$ is edge-color-avoiding connected and has $r + {\footnotesize \Big\lceil} \frac{r}{k-1} {\footnotesize \Big\rceil} = {\footnotesize \Big\lceil} \frac{k\cdot r}{k-1} {\footnotesize \Big\rceil}$ edges.
\end{proof}

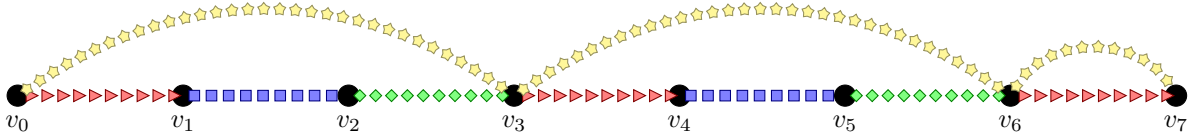
\begin{figure}[H]
\centering
\begin{tikzpicture}[scale=2.2]
 \tikzstyle{vertex}=[draw,circle,fill,minimum size=8,inner sep=0]
 \tikzset{paint/.style={draw=#1!50!black, fill=#1!50}, decorate with/.style = {decorate, decoration={shape backgrounds, shape=#1, shape size = 4.5pt, shape sep = 6pt}}}
 \tikzstyle{edge_red}=[draw, decorate with = isosceles triangle, paint = red]
 \tikzstyle{edge_blue}=[draw, decorate with = rectangle, decoration = {shape size = 4pt, shape sep = 6.5pt}, paint = blue]
 \tikzstyle{edge_green}=[draw, decorate with = diamond, paint = green]
 \tikzstyle{edge_yellow}=[draw, decorate with = star, paint = yellow]
 
 \node[vertex] (a1) at (1,0) [label={below:$v_0$}] {};
 \node[vertex] (a2) at (2,0) [label={below:$v_1$}] {};
 \node[vertex] (a3) at (3,0) [label={below:$v_2$}] {};
 \node[vertex] (a4) at (4,0) [label={below:$v_3$}] {};
 \node[vertex] (a5) at (5,0) [label={below:$v_4$}] {};
 \node[vertex] (a6) at (6,0) [label={below:$v_5$}] {};
 \node[vertex] (a7) at (7,0) [label={below:$v_6$}] {};
 \node[vertex] (a8) at (8,0) [label={below:$v_7$}] {};
 
 \draw[edge_red] (a1) -- (a2);
 \draw[edge_red] (a4) -- (a5);
 \draw[edge_red] (a7) -- (a8);
 
 \draw[edge_blue] (a2) -- (a3);
 \draw[edge_blue] (a5) -- (a6);
 
 \draw[edge_green] (a3) -- (a4);
 \draw[edge_green] (a6) -- (a7);
 
 \draw[edge_yellow, bend left=35] (a1) to (a4);
 \draw[edge_yellow, bend left=35] (a4) to (a7);
 \draw[edge_yellow, bend left=60] (a7) to (a8);
\end{tikzpicture}
\caption{An edge-color-avoiding connected graph on $8$ vertices and with minimum number of edges colored with exactly $4$ colors.}
\label{fig:edge-proof-gencon}
\end{figure}

Now we present a polynomial-time approximation algorithm for finding a courteously colored rank-preserving restriction of a matroid to a set of minimum size. To shorten the description of the algorithm, let us define the following subroutine, Subroutine IncreaseRank, which greedily adds elements to a given subset $T \subseteq S$ in a matroid $\mathcal{M} = (S,I)$ so that each newly added element increases the rank of the current subset. In particular, if $\mathcal{M}$ is a graphic matroid associated to a graph $G$ and $T \subseteq E(G)$, then this subroutine adds exactly one edge of $G$ between any two components of the graph $G' = \big( V(G), T \big)$ if such an edge exists. Let us note that in the general case, the rank of a set $T \subseteq S$ can be computed using $O \big( |S| \big)$ independence oracle calls with the \emph{greedy algorithm}\footnote{The greedy algorithm in general is used for finding a maximum-weight basis in a matroid: first, we sort the elements of the ground set by weights, and starting from the empty set, in each step we extend the already existing independent set with an element having maximum possible weight.}~\cite{book:recski}.

\begin{algorithm}[H]
\caption*{\textbf{Subroutine }IncreaseRank} \label{alg:subr}
\begin{algorithmic}[1]
 \Require A matroid $\mathcal{M}=(S,\mathcal{I})$ and a subset $T \subseteq S$.
 \Ensure A set $T' \subseteq S$ for which $r(T') = r(S)$ and $T \subseteq T'$.
 \State $T' \gets T$
 \For{$s\in S$}
     \If{$r(T') < r \big( T'\cup\{s\} \big)$}
         \State $T' \gets T'\cup \{s\}$
     \EndIf
 \EndFor
 \State \Return $T'$
\end{algorithmic}
\end{algorithm}

Now we are ready to present Algorithm~\ref{alg:matr}, which does the following. Given a courteously colored matroid $\mathcal{M} = (S, \mathcal{I})$ with $S \ne \emptyset$, colored with a color set $C$, first we pick an arbitrary basis $T$ of $\mathcal{M}$. Then, for any color $c \in C$, if the removal of the elements of color $c$ decreases the rank of the set of the so far selected elements, then using Subroutine IncreaseRank, we select some additional elements to avoid the mentioned decrease in the rank. At the end, the selected elements clearly form a courteously colored rank-preserving restriction.

\begin{algorithm}[H]
\caption{Finding courteously colored rank-preserving restrictions} \label{alg:matr}
\begin{algorithmic}[1]
 \Require A courteously colored matroid $\mathcal{M} = (S, \mathcal{I})$ with $S \ne \emptyset$, colored with a color set $C$.
 \Ensure A courteously colored rank-preserving restriction of $\mathcal{M}$.
 \State $T \gets \text{a basis of $\mathcal{M}$}$ \hfill \texttt{\textcolor{black!50}{//~Phase~1}}
 \For{$c\in C$} \hfill \texttt{\textcolor{black!50}{//~Phase~2}}
     \State $T_{\overline{c}} \gets \{ s \in T \mid \text{$s$ is not of color $c$} \}$
     \If{$r(T_{\overline{c}}) < r(S)$}
         \State $T \gets T \cup \text{IncreaseRank}(\mathcal{M}_{\overline{c}},T_{\overline{c}})$
     \EndIf
 \EndFor
 \State \Return $\mathcal{M}\big|_T$
\end{algorithmic}
\end{algorithm}

In the following theorem, we analyze Algorithm~\ref{alg:matr}.

\begin{theorem}\label{thm:edgeappr}
 Algorithm~\ref{alg:matr} is a polynomial-time $\frac{2(k-1)}{k}$-approximation algorithm for finding a courteously colored rank-preserving restriction of a courteously colored matroid --- given by an independence oracle --- whose elements are colored with exactly $k \in \mathbb{Z}_+$ colors to a set of minimum size.
\end{theorem}
\begin{proof}
 We discuss the time complexity and the correctness of the algorithm separately.
 
 \medskip
 
 \textit{Correctness.} Let $\mathcal{M} = (S, \mathcal{I})$ be the input matroid whose ground set is courteously colored with exactly $k \in \mathbb{Z}_+$ colors and let $r \colonequals r(S)$. If $k=1$, then by Theorem~\ref{thm:edge}, $r=0$ must hold, i.e., the rank of the input matroid $\mathcal{M}$ must be zero. Then for any color $c$, the rank of $\mathcal{M}_{\overline{c}}$ is also zero and the only basis of $\mathcal{M}$ is the empty set. Thus the algorithm selects the empty set in Phase~1 and does not add any elements to it in Phase~2, so the output is $T = \emptyset$. This is clearly the optimal solution.
 
 \medskip
 
 Now assume $k \ge 2$. In Phase~1, we simply select a basis of $\mathcal{M}$. Note that this step already guarantees the rank-preserving property of the output.
 
 Now we show that in Phase~2, the algorithm selects some additional elements of the ground set to ensure that the output is courteously colored. Since $\mathcal{M}$ is courteously colored, $\mathcal{M}_{\overline{c}}$ has the same rank as~$\mathcal{M}$, namely $r$, thus we can select some additional elements from the ground set of $\mathcal{M}_{\overline{c}}$ so that the obtained set of selected elements has rank $r$ as well. Therefore, at the end of Phase~2, the restriction of $\mathcal{M}$ to the so far selected elements is indeed courteously colored.
 
 
 Note that the elements of the output are not necessarily colored with exactly $k$ colors.
 
 \medskip
 
 \textit{Approximation factor.} Now we prove that the ground set of the output contains at most $\frac{2(k-1)}{k}\cdot \frac{k \cdot r}{k-1} = 2r$ elements, which is, by Theorem~\ref{thm:edge}, at most $\frac{2(k-1)}{k}$ times as many as the minimum number of elements of a courteously colored matroid whose elements are colored with at most $k$ colors, implying that Algorithm~\ref{alg:matr} is a $\frac{2(k-1)}{k}$-approximation algorithm.
 
 In Phase~1, the algorithm selects a basis $B$ of $\mathcal{M}$, which clearly consists of $r$ elements.
 
 In Phase~2, for each color $c$, if the deletion of the elements of color $c$ decreases the rank of the set of the so far selected elements, then the algorithm selects some additional elements of some colors different from $c$ to avoid this happening. More precisely, if the deletion of the elements of color $c$ decreases the rank of the set of the so far selected elements by $x_c$, then the algorithm selects $x_c$ new elements of some colors different from $c$. For any color $c$, let $y_c$ denote the number of elements of color $c$ in the basis $B$ found in the first step. Since $B$ is a basis, the deletion of $y_c$ elements from $B$ decreases its rank by exactly $y_c$. Note that during Phase~2, always some superset of $B$ is selected, which implies that the deletion of the elements of color $c$ decreases the rank of the set of the so far selected elements by at most $y_c$. Thus the algorithm selects at most $y_c$ elements for every color $c$ in Phase~2. Therefore, at most
 \[ r + \sum_{c \in C} y_c = r + |B| = 2r \]
 elements are selected when the algorithm terminates.
 
 
 \medskip
 
 \textit{Time complexity.} Next we prove that the algorithm runs in polynomial time if $\mathcal{M}$ is given by an independence oracle.
 
 In Phase~1, the algorithm selects a basis, which can be done in $O \big( |S| \big)$ time with the greedy algorithm.
 
 In Phase~2, for every color $c \in C$, we remove the elements of color $c$ from the set of the so far selected elements --- which can be done in $O \big( |S| \big)$ time ---, then we select some additional elements with the use of Subroutine IncreaseRank --- which can be done in $O \big( |S|^2 \big)$ time. Thus the algorithm takes $O \big( |C| \cdot |S|^2 \big)$ steps in Phase~2.
 
 
 Therefore, the algorithm runs in $O \big( |C| \cdot |S|^2 \big)$, i.e.\ in polynomial time.
\end{proof}

\medskip

As is clear from the proof of Theorem~\ref{thm:edgeappr}, if $k=1$ then Algorithm~\ref{alg:matr} always finds an optimal solution. In the following, we show that there exist inputs for which the approximation ratio of Algorithm~\ref{alg:matr} is exactly $\frac{2(k-1)}{k}$ if $k \ge 2$, i.e., we present some courteously colored matroids --- more precisely, some edge-color-avoiding connected graphs $G$ --- for which the output of Algorithm~\ref{alg:matr} can have exactly $2(n-1)$ edges, while the optimal solution has $\frac{k \cdot (n-1)}{k-1}$ edges. Let us define the edge-colored graph $G=(V,E)$ as follows. Let
\[ C \colonequals \{0, 1, \ldots, k-1\} \]
be the color set, let
\[ V \colonequals \{ v_0, \ldots, v_{n-1} \} \]
with $(k-1) \mid (n-1)$, and let
\[ E_i \colonequals \big\{ v_j v_{j+1} \bigm| j \in \{ 0, 1, \ldots, n-2 \} \text{ and } j \equiv i \!\!\! \pmod{k-1} \big\} \]
and
\[ E'_i \colonequals \big\{ v_j v_{j+1} \bigm| j \in \{ 0, 1, \ldots, n-2 \} \text{ and } j+1 \equiv i \!\!\! \pmod{k-1} \big\} \]
be the sets of edges of color $i$ for any $i \in \{ 0, 1, \ldots, k-2 \}$, and let
\[ E_{k-1} \colonequals \big\{ v_j v_{j+k-1} \bigm| j \in \{ 0, \ldots, n-2 \} \text{ and } j \equiv 0 \!\!\! \pmod{k-1} \big\} \]
be the set of edges of color $k-1$, and let $E = \big( \bigcup_{i \in C} E_i \big) \cup \big( \bigcup_{i \in C \setminus \{ k-1 \}} E'_i \big)$. For an example, see Figure~\ref{fig:tight-constr}.

By Theorem~\ref{thm:edge}, the subgraph $\big( V, \bigcup_{i\in C} E_i)$ is an optimal solution with $\frac{k \cdot (n-1)}{k-1}$ edges. However, the output of the algorithm can also be the subgraph $\big( V, \bigcup_{i \in C \setminus \{ k-1 \}} ( E_i \cup E'_i) \big)$. To see this, first note that the edges of both $\bigcup_{i \in C \setminus \{ k-1 \}} E_i$ and of $\bigcup_{i \in C \setminus \{ k-1 \}} E'_i$ form Hamiltonian paths $v_0 v_1 \ldots v_{n-1}$, which are disjoint with each other. Thus the algorithm can select the edges of $\bigcup_{i \in C \setminus \{ k-1 \}} E_i$ in Phase~1 --- these edges obviously form a spanning tree ---, then it can select all the edges of $\bigcup_{i \in C \setminus \{ k-1 \}} E'_i$ in Phase~2; and all the so far selected edges clearly form an edge-color-avoiding connected graph. Clearly, this output has $2(n-1)$ edges, therefore the approximation ratio in this case is indeed $\frac{2(k-1)}{k}$.

\begin{figure}[H]
\centering
\begin{tikzpicture}[scale=2.3]
 \tikzstyle{vertex}=[draw,circle,fill,minimum size=8,inner sep=0]
 \tikzset{paint/.style={draw=#1!50!black, fill=#1!50}, decorate with/.style = {decorate, decoration={shape backgrounds, shape=#1, shape size = 4.5pt, shape sep = 6pt}}}
 \tikzstyle{edge_red}=[draw, decorate with = isosceles triangle, paint = red]
 \tikzstyle{edge_blue}=[draw, decorate with = rectangle, decoration = {shape size = 4pt, shape sep = 6.5pt}, paint = blue]
 \tikzstyle{edge_green}=[draw, decorate with = diamond, paint = green]
 \tikzstyle{edge_yellow}=[draw, decorate with = star, paint = yellow]
 
 \node[vertex] (a1) at (1,0) [label={below:$v_0$}] {};
 \node[vertex] (a2) at (2,0) [label={below:$v_1$}] {};
 \node[vertex] (a3) at (3,0) [label={below:$v_2$}] {};
 \node[vertex] (a4) at (4,0) [label={below:$v_3$}] {};
 \node[vertex] (a5) at (5,0) [label={below:$v_4$}] {};
 \node[vertex] (a6) at (6,0) [label={below:$v_5$}] {};
 \node[vertex] (a7) at (7,0) [label={below:$v_6$}] {};
 
 \draw[edge_red, bend left = 35] (a1) to (a2);
 \draw[edge_red, bend left = 35] (a3) to (a4);
 \draw[edge_red, bend left = 35] (a5) to (a6);
 
 \draw[edge_red, bend right = 35] (a2) to (a3);
 \draw[edge_red, bend right = 35] (a4) to (a5);
 \draw[edge_red, bend right = 35] (a6) to (a7);
 
 \draw[edge_blue, bend left = 35] (a2) to (a3);
 \draw[edge_blue, bend left = 35] (a4) to (a5);
 \draw[edge_blue, bend left = 35] (a6) to (a7);
 
 \draw[edge_blue, bend right = 35] (a1) to (a2);
 \draw[edge_blue, bend right = 35] (a3) to (a4); 
 \draw[edge_blue, bend right = 35] (a5) to (a6);
 
 \draw[edge_green, bend left = 50] (a1) to (a3);
 \draw[edge_green, paint=green, bend left = 50] (a3) to (a5);
 \draw[edge_green, bend left = 50] (a5) to (a7);
\end{tikzpicture}
\caption{An edge-color-avoiding connected graph colored with $k=3$ colors and on $n = 7$ vertices, which contains an edge-color-avoiding connected spanning subgraph with $\frac{k \cdot (n-1)}{k-1} = 9$ edges --- for example, such a subgraph is spanned by the green (denoted by rhombi) edges and the lower red (denoted by triangles) and blue (denoted by squares) edges --- and an edge-color-avoiding connected spanning subgraph with $2(n-1) = 12$ edges --- that subgraph is spanned by the red and blue edges. The output of Algorithm~\ref{alg:matr} can be this latter subgraph, resulting in an approximation ratio of exactly $\frac{2(k-1)}{k}$.}
\label{fig:tight-constr}
\end{figure}
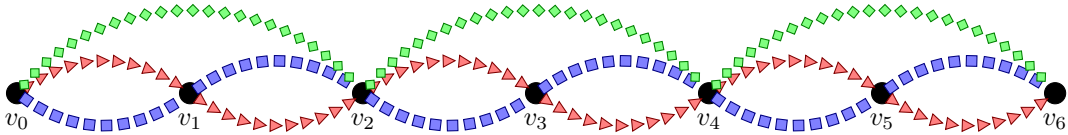 

\begin{remark} \label{remark:third_phase}
 We note that the following third phase could be added to Algorithm~\ref{alg:matr}. For every selected element $s$, we check whether for any color $c \in C$, the deletion of $s$ and all the selected elements of color $c$ decreases the rank of the set of the so far selected elements. If yes, then we keep $s$ in the set of the so far selected elements, if no, then we deselect $s$. Unfortunately, this modification does not give a better approximation factor: it is not difficult to see that for the above described example, no edges get deselected in this new phase.
\end{remark}

\begin{remark}
 Let us point out that the following modification of Subroutine IncreaseRank runs in $O \big( |S| \big)$ time --- compared to the current version which runs in $O \big( |S|^2 \big)$ time. Let us assign weight 0 to the elements of $T$ and weight 1 to all the other elements, then find a minimum-weight basis in the matroid with the greedy algorithm. Clearly, the weight of a minimum-weight basis $B$ is $r(S) - r(T)$, thus $T' \colonequals B \cup T$ has the desired properties: $T \subseteq T'$ and $r(T') = r(S)$ and $|T'| = |T| + r(S) - r(T)$. Using this modified subroutine IncreaseRank, Algorithm~\ref{alg:matr} is still a $\frac{2(k-1)}{k}$-approximation algorithm and has a running time $O \big( |C| \cdot |S| \big)$.
\end{remark}

The proof of Theorem~\ref{thm:edgeappr} also implies that a courteously colored matroid $\mathcal{M}$ which does not contain any courteously colored rank-preserving restriction to a proper subset of its ground set has at most $2r$ elements. In the following, we show that this upper bound is tight. Let $G=(V,E)$ denote the underlying graph of the courteously colored matroid $\mathcal{M}$, and let us define $G$ as follows. Let
\[ C \colonequals \{0, 1, \ldots, k-1\} \]
be the color set, let
\[ V \colonequals \{ v_0, \ldots, v_{n-1} \} \text{,} \]
and let
\[ E_i \colonequals \big\{ v_j v_{j+1} \bigm| j \in \{ 0, 1, \ldots, n-2 \} \text{ and } j \equiv i \!\!\! \pmod{k-1} \big\} \]
and
\[ E'_i \colonequals \big\{ v_j v_{j+1} \bigm| j \in \{ 0, 1, \ldots, n-2 \} \text{ and } j+1 \equiv i \!\!\! \pmod{k-1} \big\} \]
be the sets of edges of color $i$ for any $i \in C$, and let $E = \bigcup_{i \in C} (E_i \cup E'_i)$. For an example, see Figure~\ref{fig:edgemax}.


\begin{figure}[H]
\centering
\begin{tikzpicture}[scale=2.25]
 \tikzstyle{vertex}=[draw,circle,fill,minimum size=8,inner sep=0]
 \tikzset{paint/.style={draw=#1!50!black, fill=#1!50}, decorate with/.style = {decorate, decoration={shape backgrounds, shape=#1, shape size = 4.5pt, shape sep = 6pt}}}
 \tikzstyle{edge_red}=[draw, decorate with = isosceles triangle, paint = red]
 \tikzstyle{edge_blue}=[draw, decorate with = rectangle, decoration = {shape size = 4pt, shape sep = 6.5pt}, paint = blue]
 \tikzstyle{edge_green}=[draw, decorate with = diamond, paint = green]
 \tikzstyle{edge_yellow}=[draw, decorate with = star, paint = yellow]
 
 \node[vertex] (a1) at (1,0) [label={below:$v_0$}] {};
 \node[vertex] (a2) at (2,0) [label={below:$v_1$}] {};
 \node[vertex] (a3) at (3,0) [label={below:$v_2$}] {};
 \node[vertex] (a4) at (4,0) [label={below:$v_3$}] {};
 \node[vertex] (a5) at (5,0) [label={below:$v_4$}] {};
 \node[vertex] (a6) at (6,0) [label={below:$v_5$}] {};
 \node[vertex] (a7) at (7,0) [label={below:$v_6$}] {};
 \node[vertex] (a8) at (8,0) [label={below:$v_7$}] {};
 
 \draw[edge_red, bend left = 40] (a1) to (a2);
 \draw[edge_red, bend left = 40] (a5) to (a6);
 \draw[edge_red, bend right = 40] (a4) to (a5);
 
 \draw[edge_blue, bend left = 40] (a2) to (a3);
 \draw[edge_blue, bend left = 40] (a6) to (a7);
 \draw[edge_blue, bend right = 40] (a1) to (a2);
 \draw[edge_blue, bend right = 40] (a5) to (a6);
 
 \draw[edge_green, bend left = 40] (a3) to (a4);
 \draw[edge_green, bend left = 40] (a7) to (a8);
 \draw[edge_green, bend right = 40] (a2) to (a3);
 \draw[edge_green, bend right = 40] (a6) to (a7);
 
 \draw[edge_yellow, bend left=40] (a4) to (a5);
 \draw[edge_yellow, bend right=40] (a3) to (a4);
 \draw[edge_yellow, bend right=40] (a7) to (a8);
\end{tikzpicture}
\caption{An edge-color-avoiding connected graph on $n = 8$ vertices and with $2(n-1) = 14$ edges colored with exactly $k = 4$ colors, and having the property that none of the edges can be removed such that the graph remains edge-color-avoiding connected.}
\label{fig:edgemax}
\end{figure}
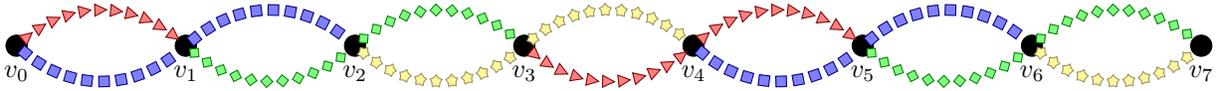 

Therefore, we obtained the following result.

\begin{corollary} \label{corollary}
 Let $\mathcal{M} = (S, \mathcal{I})$ be a courteously colored matroid with rank $r$ such that $\mathcal{M} \setminus \{ s \}$ is not courteously colored for all $s \in S$. Then $|S| \le 2r$ and this upper bound is sharp.
\end{corollary}

Let us point out that the edge-colored graphs we presented for showing that the approximation factor of Algorithm~\ref{alg:matr} is exactly $\frac{2(k-1)}{k}$ contain both a construction with minimum possible number of elements --- for this construction, see the proof of Theorem~\ref{thm:edge} --- and a construction with maximum possible number of elements --- see the construction above.

\begin{remark} \label{remark:simple_alg}
 Using Theorem~\ref{thm:edge} and Corollary~\ref{corollary}, one can design a simpler polynomial-time $\frac{2(k-1)}{k}$-approximation algorithm for finding courteously colored rank-preserving restrictions of a matroid $\mathcal{M}$ to a set of minimum size: one by one greedily delete those elements of the ground set from $\mathcal{M}$ after whose deletion the obtained matroid is courteously colored and has the same rank as $\mathcal{M}$. By the same reasoning as that for Algorithm~\ref{alg:matr}, we get that there exist inputs for which the approximation ratio of this greedy algorithm is exactly $\frac{2(k-1)}{k}$, and the running time of this algorithm is $O \big( |C| \cdot |S|^2 \big)$.
\end{remark}

\subsection{Approximation algorithm for edge-color-avoiding connected graphs}

As we observed before Definition~\ref{def:courteous_matroid}, a graphic matroid is courteously colored if and only if each component of the corresponding graph is edge-color-avoiding connected, which implies that the underlying graph of a courteously colored graphic matroid is not necessarily edge-color-avoiding connected. Therefore, we also present a separate polynomial-time approximation algorithm for finding edge-color-avoiding connected spanning subgraphs with minimum number of edges. Note that by Corollary~\ref{cor:ECA-subgraph_NP-complete}, this problem is NP-hard.

To simplify the description of the algorithm, we introduce the following simple subroutines. The subroutine \textbf{Graph}$(V,E)$ creates a graph with vertex set $V$ and edge set $E$, the subroutine \textbf{SpanningTree}$(G)$ returns the edges of some arbitrary spanning tree of a connected graph $G$, and the subroutine \textbf{ConnectedComponents}$(G)$ returns the family of the vertex sets of the connected components of $G$. The subroutine \textbf{ContractVertices}$(G, \mathcal{W})$, whose inputs are a graph $G$ and a partition $\mathcal{W}$ of its vertex set, returns a graph which can be obtained from $G$ by contracting the vertices of each set $W \in \mathcal{W}$ into a single vertex, and two distinct vertices of the obtained graph is connected by a colored edge if and only if there is an edge of the same color between the corresponding sets in the original graph. Finally, the subroutine \textbf{BeforeContraction}$(G, H, E')$, whose inputs are a graph $G$, and a graph $H$ that is obtained from $G$ by contracting some of its vertices, and an edge set $E' \subseteq E(H)$, returns a set of edges in $G$ corresponding to~$E'$.

Now we are ready to present Algorithm~\ref{alg:edge}, which does the following. Given an edge-color-avoiding connected graph $G = (V, E)$ colored with a color set $C$, first we pick an arbitrary spanning tree of $G$. Then for any color $c \in C$, if the removal of the edges of color $c$ disconnects the graph of the so far selected edges, then we contract the components of this latter graph, and we select some edges in the original graph which correspond to the edges of an arbitrary spanning tree in the contracted graph. At the end, the selected edges clearly form an edge-color-avoiding connected spanning subgraph.

\begin{algorithm}[H]
\caption{Finding edge-color-avoiding connected spanning subgraphs} \label{alg:edge}
\begin{algorithmic}[1]
 \Require An edge-color-avoiding connected graph $G = (V, E)$ colored with a color set $C$.
 \Ensure An edge-color-avoiding connected spanning subgraph $G'$ of $G$.
 \State $E' \gets \textrm{SpanningTree}(G)$ \hfill \texttt{\textcolor{black!50}{//~Phase~1}}
 \State $G' \gets \textrm{Graph}(V, E')$
 \For{$c\in C$} \hfill \texttt{\textcolor{black!50}{//~Phase~2}}
     \If{$G'_{\overline{c}}$ is not connected}
         \State $\mathcal{W} \gets \textrm{ConnectedComponents}(G'_{\overline{c}})$
         \State $H \gets \textrm{ContractVertices}(G_{\overline{c}},\mathcal{W})$
         \State $E' \gets E' \cup \textrm{BeforeContraction} \big( G, \; H, \; \textrm{SpanningTree}(H) \big)$
         \State $G' \gets \text{Graph}(V, E')$
     \EndIf
 \EndFor
 \State \Return $G'$
\end{algorithmic}
\end{algorithm}

Similarly to Algorithm~\ref{alg:matr}, at the end of Algorithm~\ref{alg:edge} we can greedily deselect those edges whose removal from the graph of the so far selected edges does not destroy the edge-color-avoiding property of this graph.

Analogously to Theorem~\ref{thm:complexity}, we can obtain the following.

\begin{theorem} \label{thm:edgeappr2}
 Algorithm~\ref{alg:edge} is a polynomial-time $\frac{2(k-1)}{k}$-approximation algorithm for the problem of finding an edge-color-avoiding connected spanning subgraph with minimum number of edges in a graph whose edges are colored with exactly $k \in \mathbb{Z}_+$ colors. Moreover, there exist inputs for which the approximation ratio is exactly $\frac{2(k-1)}{k}$.
\end{theorem}

By simply using the term ``spanning tree'' as opposed to ``basis'', the same proof as that of Theorem~\ref{thm:edgeappr} applies to prove Theorem~\ref{thm:edgeappr2}, but for a detailed proof, see~\cite{diplomamunka}.

\subsection{Approximation algorithm for vertex-color-avoiding connected graphs}

In the following, we give a polynomial-time approximation algorithm for finding a vertex-color-avoiding connected spanning subgraph with minimum number of edges. But before that, we give a lower bound on the number of edges in a vertex-color-avoiding connected graph.

\begin{theorem} \label{thm:vertex}
 Let $G$ be a vertex-color-avoiding connected graph on $n$ vertices colored with exactly $k \in \mathbb{Z}_+$ colors. Then
 \[ \big| E(G) \big| \ge \begin{cases}
                          n-1 & \text{if $k \le 2$,} \\
                          n & \text{if $k \ge 3$,}
                         \end{cases} \]
 and this lower bound is sharp.
\end{theorem}

To prove this theorem, we use the following lemma. By a \emph{cut-vertex}, we mean a vertex whose removal disconnects the graph.

\begin{lemma} \label{lem:components}
 If $v$ is a cut-vertex in a vertex-color-avoiding connected graph $G$, then all but one component of $G - v$ consist only of vertices of the color of $v$.
\end{lemma}
\begin{proof}
 Let $G$ be a vertex-color-avoiding connected graph with a cut-vertex $v$ of color $c$. Suppose to the contrary that there are at least two components in $G-v$ which contain vertices of some colors different from $c$. Let $v_1$ and $v_2$ be two vertices such that their colors are different from $c$ and they are in different components in $G-v$. Then $v_1$ and $v_2$ are not vertex-$c$-avoiding connected, which is a contradiction.
\end{proof}

Now we prove Theorem~\ref{thm:vertex}.

\begin{proof}[Proof of Theorem~\ref{thm:vertex}]

 

 First, let $k \le 2$. By the definition of vertex-color-avoiding connectivity, $G$ must be connected, so it must have at least $n-1$ edges. Clearly, any $n$-vertex tree whose every vertex except at most one of its leaves is of the same color achieves this lower bound.
 
 \medskip
 
 Now let $k \ge 3$. First, we show that any vertex-color-avoiding connected graph $G$ on $n$ vertices colored with exactly $k \ge 3$ colors must have at least $n$ edges. Suppose to the contrary that there exists such a graph with $n-1$ edges. By the definition of vertex-color-avoiding connectivity, $G$ must be connected. Thus $G$ is a tree. Clearly, $n \ge k \ge 3$, thus $G$ has at least two leaves. If the non-leaf vertices are all of the same color $c$, then there exist at least two leaves $u$ and $v$ of some colors different from $c$. Then $u$ and $v$ are not vertex-$c$-avoiding connected, which is a contradiction. Therefore, there exist two adjacent non-leaf vertices $v_1$ and $v_2$ of colors $c_1$ and $c_2$, respectively, where $c_1 \ne c_2$. By Lemma~\ref{lem:components}, since $v_1$ is a cut-vertex, all components of $G-v_1$ but one contain vertices only of color $c_1$. The remaining one component must be the one which contains $v_2$. The same can be told about $v_2$ and $c_2$. It is easy to see that by the adjacency of $v_1$ and $v_2$ this means that all of the vertices must be of color $c_1$ or $c_2$. This contradicts the fact that the graph is colored with exactly $k \ge 3$ colors. Thus any vertex-color-avoiding connected graph on $n$ vertices colored with exactly $k \ge 3$ colors must have at least $n$ edges.
 
 Now we present a vertex-color-avoiding connected graph $G=(V,E)$ on $n$ vertices colored with exactly $k \ge 3$ colors that has $n$ edges. Let
 \[ V = \{ v_1, v_2, \ldots, v_n \} \]
 with $n\geq 3$ and let
 \[ E = \{ v_1v_2, \, v_2v_3, \, \ldots, \, v_{n-1}v_n, \, v_nv_1\} \text{.} \]
 Let the color function $c$ be
 \[ c \colon V(G) \to \{ 1, 2, \ldots, k \} \qquad v \mapsto \begin{cases}
i & \text{if $v =  v_i$ and $i \le k$,} \\
k & \text{otherwise.}
\end{cases} \]
 For an example, see Figure~\ref{fig:vertexproof}. It is not difficult to see that with this vertex-coloring, $G$ is vertex-color-avoiding connected and has $n$ edges.
\end{proof}

\begin{figure}[H]
\centering
\begin{tikzpicture}[scale=1]
 \tikzstyle{vertex_red}=[draw, shape=regular polygon, regular polygon sides=3, fill=red, minimum size=12pt, inner sep=0]
 \tikzstyle{vertex_blue}=[draw, shape=rectangle, fill=blue, minimum size=8pt, inner sep=0]
 \tikzstyle{vertex_green}=[draw, shape=diamond, fill=green, minimum size=11pt, inner sep=0]
 \tikzstyle{vertex_yellow}=[draw, shape=star, fill=yellow, minimum size=11pt, inner sep=0]
 
 \node[vertex_blue] (a1) at (0:1) [label={right:$v_4$}] {};
 \node[vertex_blue] (a2) at (60:1) [label={above:$v_5$}] {};
 \node[vertex_blue] (a3) at (120:1) [label={above:$v_6$}] {};
 \node[vertex_red] (a4) at (180:1) [label={left:$v_1$}] {};
 \node[vertex_yellow] (a5) at (240:1) [label={below:$v_2$}] {};
 \node[vertex_green] (a6) at (300:1) [label={below:$v_3$}] {};
 \draw[very thick] (a1) -- (a2) -- (a3) -- (a4) -- (a5) -- (a6) -- (a1);
\end{tikzpicture}
\caption{A vertex-color-avoiding connected graph on $n = 6$ vertices colored with exactly $k = 4$ colors and with $n = 6$ edges.}
\label{fig:vertexproof}
\end{figure}
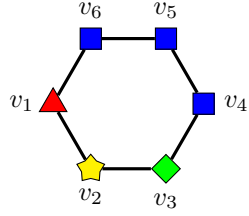 

Now we present a polynomial-time approximation algorithm for finding a vertex-color-avoiding connected spanning subgraph with minimum number of edges. This algorithm, presented as Algorithm~\ref{alg:vertex}, follows the same scheme as Algorithm~\ref{alg:edge} --- the only difference is that here the vertices are colored. Given a vertex-color-avoiding connected graph $G = (V, E)$ colored with a color set $C$, first we pick an arbitrary spanning tree of $G$. Then for any color $c \in C$, if the removal of the vertices of color $c$ disconnects the graph of the so far selected edges, then we contract the components of the graph whose vertex set consists of the vertices which have colors different from $c$ and whose edge set consists of the selected edges. Then we select some edges in the original graph which correspond to the edges of an arbitrary spanning tree in this contracted graph.
At the end, the selected edges clearly form a vertex-color-avoiding connected spanning subgraph.

\begin{algorithm}[H]
\caption{Finding vertex-color-avoiding connected spanning subgraphs} \label{alg:vertex}
\begin{algorithmic}[1]
 \Require A vertex-color-avoiding connected graph $G = (V, E)$ colored with a color set $C$.
 \Ensure A vertex-color-avoiding connected spanning subgraph $G'$ of $G$.
 \State $E' \gets \textrm{SpanningTree}(G)$ \hfill \texttt{\textcolor{black!50}{//~Phase~1}}
 \State $G' \gets \textrm{Graph$(V, E')$}$
 \For{$c\in C$} \hfill \texttt{\textcolor{black!50}{//~Phase~2}}
     \If{$G'_{\overline{c}}$ is not connected}
         \State $\mathcal{W} \gets \textrm{ConnectedComponents}(G'_{\overline{c}})$
         \State $H \gets \textrm{ContractVertices}(G_{\overline{c}},\mathcal{W})$
         \State $E' \gets E' \cup \textrm{BeforeContraction} \big( G, \; H, \; \text{SpanningTree}(H) \big)$
         \State $G' \gets \text{Graph$(V, E')$}$
     \EndIf
 \EndFor
 \State \Return $G'$
\end{algorithmic}
\end{algorithm}



Let us note that similarly to Algorithms~\ref{alg:matr} and~\ref{alg:edge}, at the end of Algorithm~\ref{alg:vertex} we can greedily deselect those edges whose removal from the graph of the so far selected edges does not destroy the vertex-color-avoiding property of this graph.

In the following theorem, we analyze Algorithm~\ref{alg:vertex}.

\begin{theorem} \label{thm:vertexappr}
 Algorithm~\ref{alg:vertex} is a polynomial-time $2$-approximation algorithm for the problem of finding a vertex-color-avoiding connected spanning subgraph with minimum number of edges.
\end{theorem}
\begin{proof}
 We discuss the time complexity and the correctness of the algorithm separately.
 
 \medskip
 
 \textit{Correctness.} Let $G=(V,E)$ be the input graph whose vertices are colored with exactly $k \in \mathbb{Z}_+$ colors.
 
 Since $G$ is vertex-color-avoiding connected, it is connected as well, thus it has a spanning tree. In Phase~1, the algorithm selects the edges of a spanning tree.
 
 Now we show that in Phase~2, the algorithm selects some additional edges to ensure the vertex-color-avoiding connectivity of the output. Since $G$ is vertex-color-avoiding connected, the graph $G_{\overline{c}}$ is connected for any color $c$ and any graph obtained from $G_{\overline{c}}$ by contracting some disjoint sets of vertices is also connected, thus it has a spanning tree. Therefore, by selecting of the edges of such a spanning tree, the obtained graph of the selected edges becomes vertex-color-avoiding connected.
 
 \medskip
 
 \textit{Approximation factor.} Now we prove that the output has at most $2n - 3$ edges, which is, by Theorem~\ref{thm:vertex}, at most twice as many as the minimum number of edges of a vertex-color-avoiding connected graph, implying that Algorithm~\ref{alg:vertex} is a 2-approximation algorithm.
 
 In Phase~1, the algorithm selects the edges of a spanning tree $T$, which clearly has $n-1$ edges.
 
 In Phase~2, for each color $c$, if the removal of the vertices of color $c$ disconnects the graph of the so far selected edges, then the algorithm selects some additional edges to avoid this happening. More precisely, if the removal of the vertices of color $c$ from the graph of the selected edges leaves $x_c$ connected components, then the algorithm selects $x_c - 1$ new edges of some colors different form $c$. Since $T$ is a spanning tree, the deletion of a vertex $v$ from it creates exactly $d_T(v)$ components, where $d_T(v)$ denotes the degree of $v$ in $T$. Thus the removal of all the vertices of color $c$ from the graph of the so far selected edges disconnects this graph into at most
 \[ 1 + \sum_{\substack{v \in V: \\ c(v)=c}} \big( d_T(v)-1 \big) \]
 components, where $c(v)$ denotes the color of a vertex $v$. Therefore, at the end of Phase~2, at most
 \[ (n-1) + \sum_{c \in C} \sum_{\substack{v \in V: \\ c(v)=c}} \big( d_T(v) - 1 \big) = (n-1) + \sum_{v \in V} d_T(v) - n = (n-1) + 2(n-1) - n = 2n-3 \]
 edges are selected.
 
 \medskip
 
 \textit{Time complexity.} Similarly to Algorithm~\ref{alg:edge}, Algorithm~\ref{alg:vertex} also runs in $O \big( |C| \cdot |V|^2 \big)$ time.
\end{proof}

\medskip

Note that if $k=1$, then any tree is vertex-color-avoiding connected; accordingly, in this case Algorithm~\ref{alg:vertex} does not select any edges in Phase~2, and thus the output has $n-1$ edges, which by Theorem~\ref{thm:vertex} is an optimal solution. Then, it is natural to ask whether a more technical analysis can establish a better approximation factor for Algorithm~\ref{alg:vertex} if $k \ge 2$. The following examples show that this is not the case. More precisely, for $k \ge 2$, we present some vertex-color-avoiding connected graphs for which the output of Algorithm~\ref{alg:vertex} can have exactly $2n-3$ edges, while the optimal solution has $n-1$ edges if $k=2$ and has $n$ edges otherwise. 
 
First, assume $k=2$ and let us define the vertex-colored graph $G=(V,E)$ as follows. Let $C \colonequals \{ 0, 1 \}$ be the color set and let
\[ V \colonequals \{ v_0, \ldots, v_{n-1} \} \text{,} \]
and for any $j \in \{ 0, 1, \ldots, n-1 \}$, let $v_j$ be of color $j \pmod{k}$. Let
\[ E' \colonequals \big\{ v_j v_{j+2} \bigm| j \in \{ 0, 1, \ldots, n-3 \} \big\} \cup \{ v_0 v_1 \} \]
and
\[ E'' \colonequals \big\{ v_j v_{j+1} \bigm| j \in \{ 0, 1, \ldots, n-2 \} \big\} \cup \big\{ v_j v_{j+2} \bigm| j \in \{ 0, 1, \ldots, n-3 \} \big\} \text{,} \]
and let $E = E' \cup E''$. For an example, see Figure~\ref{fig:vertex-algo}.

By Theorem~\ref{thm:vertex}, the subgraph $(V, E')$ is an optimal solution with $n-1$ edges. However, the output of the algorithm can also be the subgraph $(V, E'')$: the algorithm can select the edges of $\big\{ v_j v_{j+1} \bigm| j \in \{ 0, 1, \ldots, n-2 \} \big\}$ in Phase~1 (these edges obviously form a Hamiltonian path, and thus a spanning tree of $G$), then it can select all the edges of $\big\{ v_j v_{j+2} \bigm| j \in \{ 0, 1, \ldots, n-3 \}$ in Phase~2 (and all the so far selected edges clearly form a vertex-color-avoiding connected graph). Clearly, this output has $2n-3$ edges.
 
Now assume $k \ge 3$ and let us define the vertex-colored graph $G=(V,E)$ as follows. Let
\[ C \colonequals \{ 0, 1, \ldots, k-1 \} \]
be the color set and let
\[ V \colonequals \{ v_0, \ldots, v_{n-1} \} \text{,} \]
and for any $j \in \{ 0, 1, \ldots, n-1 \}$, let $v_j$ be of color $j \pmod{k}$. Let
\[ E' \colonequals \big\{ v_j v_{j+k} \bigm| j \in \{ 0, 1, \ldots, n-k-1 \} \big\} \cup \big\{ v_j v_{\ell} \bigm| j \in \{ n-k, \ldots, n-1 \} \text{ and } \ell = j+1 \!\!\! \pmod{k} \big\} \]
and
\[ E'' \colonequals \big\{ v_j v_{j+1} \bigm| j \in \{ 0, 1, \ldots, n-2 \} \big\} \cup \big\{ v_j v_{j+2} \bigm| j \in \{ 0, 1, \ldots, n-3 \} \big\} \text{,} \]
and let $E = E' \cup E''$. For an example, see Figure~\ref{fig:vertex-algo}.

By Theorem~\ref{thm:vertex}, the subgraph $(V, E')$ is (such a Hamiltonian cycle in $G$, along which the vertices of the same color are consecutive, thus it is) an optimal solution with $n$ edges. However, the output of Algorithm~\ref{alg:vertex} can also be the subgraph $(V, E'')$: the algorithm can select the edges of $\big\{ v_j v_{j+1} \bigm| j \in \{ 0, 1, \ldots, n-2 \} \big\}$ in Phase~1 (these edges obviously form a Hamiltonian path, and thus a spanning tree of $G$), then it can select all the edges of $\big\{ v_j v_{j+2} \bigm| j \in \{ 0, 1, \ldots, n-3 \}$ in Phase~2 (and all the so far selected edges clearly form a vertex-color-avoiding connected graph). Clearly, this output has $2n-3$ edges.

\begin{figure}[H]
\centering
\begin{tikzpicture}[scale=1.5]
 \tikzstyle{vertex_red}=[draw, shape=regular polygon, regular polygon sides=3, fill=red, minimum size=12pt, inner sep=0]
 \tikzstyle{vertex_blue}=[draw, shape=rectangle, fill=blue, minimum size=8pt, inner sep=0]
 \tikzstyle{vertex_green}=[draw, shape=diamond, fill=green, minimum size=11pt, inner sep=0]
 \tikzstyle{vertex_yellow}=[draw, shape=star, fill=yellow, minimum size=11pt, inner sep=0]
 
 \begin{scope}[rotate={90+360/7}]
 \node[vertex_blue] (a0) at (0*360/7:1) [label={[shift={($(0,-0.4)+(90+360/7:0.5)$)}] $v_0$}] {};
 \node[vertex_red] (a1) at (1*360/7:1) [label={[shift={($(0,-0.4)+(90+2*360/7:0.85)$)}] $v_1$}] {};
 \node[vertex_blue] (a2) at (2*360/7:1) [label={[shift={($(0,-0.4)+(90+3*360/7:0.5)$)}] $v_2$}] {};
 \node[vertex_red] (a3) at (3*360/7:1) [label={[shift={($(0,-0.4)+(90+4*360/7:0.85)$)}] $v_3$}] {};
 \node[vertex_blue] (a4) at (4*360/7:1) [label={[shift={($(0,-0.4)+(90+5*360/7:0.5)$)}] $v_4$}] {};
 \node[vertex_red] (a5) at (5*360/7:1) [label={[shift={($(0,-0.4)+(90+6*360/7:0.85)$)}] $v_5$}] {};
 \node[vertex_blue] (a6) at (6*360/7:1) [label={[shift={($(0,-0.4)+(90+7*360/7:0.5)$)}] $v_6$}] {};
 \draw[very thick] (a0) -- (a1) -- (a2) -- (a3) -- (a4) to [bend left=15] (a5) to [bend left=15] (a6);
 \draw[very thick] (a1) -- (a3) to [bend left=15] (a5);
 \draw[very thick] (a0) to [bend left=15] (a2) to [bend left=15] (a4) to [bend left=15] (a6);
 \draw[very thick,dashed] (a0) to [bend right=90, looseness=1.5] (a2) to [bend right=90, looseness=1.5] (a4) to [bend right=90, looseness=1.5] (a6);
 \draw[very thick,dashed] (a0) to [bend right=45, looseness=1] (a1) to [bend right=90, looseness=2] (a3) to [bend right=90, looseness=2] (a5);
 \end{scope}

 \begin{scope}[shift={(5,0)}, rotate={90+360/9}]
 \node[vertex_blue] (a0) at (0*360/9:1) [label={above left:$v_0$}] {};
 \node[vertex_red] (a1) at (1*360/9:1) [label={left:$v_1$}] {};
 \node[vertex_yellow] (a2) at (2*360/9:1) [label={below left:$v_2$}] {};
 \node[vertex_green] (a3) at (3*360/9:1) [label={below:$v_3$}] {};
 \node[vertex_blue] (a4) at (4*360/9:1) [label={below:$v_4$}] {};
 \node[vertex_red] (a5) at (5*360/9:1) [label={below right:$v_5$}] {};
 \node[vertex_yellow] (a6) at (6*360/9:1) [label={right:$v_6$}] {};
 \node[vertex_green] (a7) at (7*360/9:1) [label={above right:$v_7$}] {};
 \node[vertex_blue] (a8) at (8*360/9:1) [label={above:$v_8$}] {};
 \draw[very thick] (a0) -- (a1) -- (a2) -- (a3) -- (a4) -- (a5) -- (a6) -- (a7) -- (a8);
 \draw[very thick] (a1) -- (a3) -- (a5) -- (a7);
 \draw[very thick] (a0) -- (a2) -- (a4) -- (a6) -- (a8);
 \draw[very thick,dashed] (a0) -- (a4) -- (a8) -- (a1) -- (a5) -- (a2) -- (a6) -- (a3) -- (a7) -- (a0);
 \end{scope}
\end{tikzpicture}
\caption{On the left, a vertex-color-avoiding connected graph colored with $k=2$ colors and on $n = 7$ vertices, which contains a vertex-color-avoiding connected spanning subgraph with $n-1 = 6$ edges --- for example, the dashed edges form such a subgraph --- and a vertex-color-avoiding connected spanning subgraph with $2n-3 = 11$ edges --- that subgraph is spanned by the remaining (i.e.\ solid) edges. On the right, a vertex-color-avoiding connected graph colored with $k=4$ colors and on $n = 9$ vertices, which contains a vertex-color-avoiding connected spanning subgraph with $n = 9$ edges --- for example, the dashed edges form such a subgraph --- and a vertex-color-avoiding connected spanning subgraph with $2n-3 = 15$ edges --- that subgraph is spanned by the remaining (i.e.\ solid) edges. In both cases, the output of Algorithm~\ref{alg:vertex} can be the subgraph of the solid edges.}
\label{fig:vertex-algo}
\end{figure}
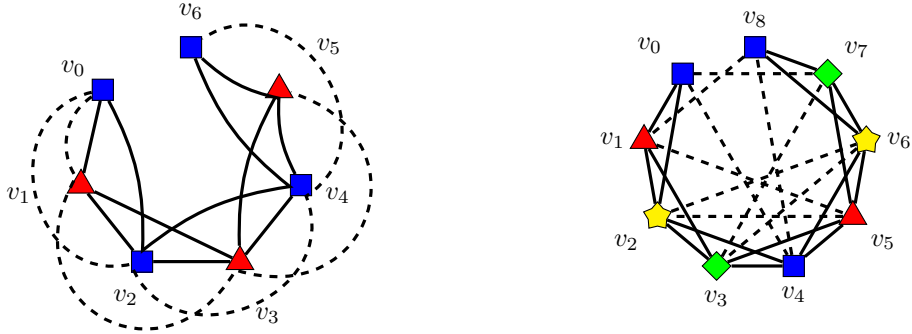

\begin{remark} \label{remark:vertex}
 If $G=(V,E)$ is a vertex-color-avoiding connected graph on $n$ vertices colored with exactly $k=2$ colors, then we can always find a vertex-color-avoiding connected spanning subgraph with $n-1$ edges --- which, by Theorem~\ref{thm:vertex} is then an optimal solution --- in polynomial time as follows. First, note that since $G$ is vertex-color-avoiding connected and $k=2$, the subgraph spanned by any of the two colors $c \in \{ 1, 2 \}$ must be connected, and thus it must contain a spanning tree $T_c$. Let $e$ be an edge whose both endvertices are of different colors; note that by the definition of vertex-color-avoiding connectivity, $G$ must be connected, and thus such an edge $e$ must exist. Then the edges of $E(T_1) \cup E(T_2) \cup \{ e \}$ form a vertex-color-avoiding spanning subgraph of $G$ with $n-1$ edges.
\end{remark}

Let us consider those vertex-color-avoiding connected graphs, which do not contain any vertex-color-avoiding connected spanning subgraph except themselves. By the proof of Theorem~\ref{thm:vertexappr}, such a graph $G=(V,E)$ on $n$ vertices has at most $2n-3$ edges. In addition, by the discussion after the proof of Theorem~\ref{thm:vertexappr} and by Remark~\ref{remark:vertex}, if $k \le 2$, then such a graph $G$ is a tree. In the following, we show that if $k = 3$, then the upper bound $|E| \le 2n-3$ is tight. To see this, let $G = (V,E)$ be defined as follows. Let
\[ C \colonequals \{ 0, 1, 2 \} \]
be the color set and let
\[ V \colonequals \{ v_0, \ldots, v_{n-1} \} \text{,} \]
and for any $j \in \{ 0, 1, \ldots, n-1 \}$, let $v_j$ be of color $j \pmod{k}$, and let
\[ E \colonequals \big\{ v_j v_{j+1} \bigm| j \in \{ 0, 1, \ldots, n-2 \} \big\} \cup \big\{ v_j v_{j+2} \bigm| j \in \{ 0, 1, \ldots, n-3 \} \big\} \text{.} \]
For an example, see Figure~\ref{fig:vertex-sharp}.

\begin{figure}[H]
\centering
\begin{tikzpicture}[scale=1.2]
 \tikzstyle{vertex_red}=[draw, shape=regular polygon, regular polygon sides=3, fill=red, minimum size=12pt, inner sep=0]
 \tikzstyle{vertex_blue}=[draw, shape=rectangle, fill=blue, minimum size=8pt, inner sep=0]
 \tikzstyle{vertex_green}=[draw, shape=diamond, fill=green, minimum size=11pt, inner sep=0]
 \tikzstyle{vertex_yellow}=[draw, shape=star, fill=yellow, minimum size=11pt, inner sep=0]
\begin{scope}[rotate={135}]
 \node[vertex_blue] (a1) at (0:1) [label={above left:$v_0$}] {};
 \node[vertex_red] (a2) at (45:1) [label={left:$v_1$}] {};
 \node[vertex_yellow] (a3) at (90:1) [label={below left:$v_2$}] {};
 \node[vertex_blue] (a4) at (135:1) [label={below:$v_3$}] {};
 \node[vertex_red] (a5) at (180:1) [label={below right:$v_4$}] {};
 \node[vertex_yellow] (a6) at (225:1) [label={right:$v_5$}] {};
 \node[vertex_blue] (a7) at (270:1) [label={above right:$v_6$}] {};
 \node[vertex_red] (a8) at (315:1) [label={above:$v_7$}] {};
 \draw[very thick] (a1) -- (a2) -- (a3) -- (a4) -- (a5) -- (a6) -- (a7) -- (a8);
 \draw[very thick] (a1) -- (a3) -- (a5) -- (a7);
 \draw[very thick] (a2) -- (a4) -- (a6) -- (a8);
\end{scope}
\end{tikzpicture}
\caption{A vertex-color-avoiding connected graph on $n = 8$ vertices colored with exactly $k = 3$ colors having $2n-3 = 13$ edges and having the property that none of the edges can be removed such that the graph remains vertex-color-avoiding connected.}
\label{fig:vertex-sharp}
\end{figure}
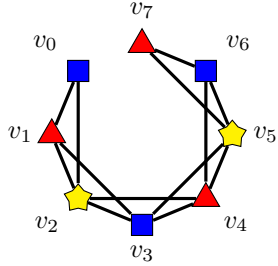

Therefore, we obtained the following.

\begin{corollary}\label{corollary:vertex}
 Let $G=(V,E)$ be a vertex-color-avoiding connected graph on $n$ vertices such that $G-e$ is not vertex-color-avoiding connected for any $e \in E$. Then $|E| \le 2n-3$ and if the vertices are colored with exactly $k=3$ colors, then this upper bound is sharp. Moreover, if $k=1$ or $k=2$, then $|E| = n-1$.
\end{corollary}

By Corollary~\ref{corollary:vertex}, a simple greedy algorithm (i.e.\ an analogous version of the one described in Remark~\ref{remark:simple_alg}) also performs in polynomial time and with the same approximation factor as Algorithm~\ref{alg:vertex}.

\subsection{Approximation\hspace{1pt} algorithm\hspace{1pt} for\hspace{1pt} internally\hspace{1pt} vertex-color-avoiding\hspace{1pt} connected graphs}

Now we give a polynomial-time approximation algorithm for finding an internally vertex-color-avoiding connected spanning subgraph with minimum number of edges. But before that, we give a lower bound on the number of edges in an internally vertex-color-avoiding connected graph.

\begin{theorem} \label{thm:ivertex}
 Let $G$ be an internally vertex-color-avoiding connected graph on $n$ vertices colored with exactly $k$ colors. Then
 \[ \big| E(G) \big| \ge \begin{cases}
                          \displaystyle \binom{n}{2} & \text{if $k = 1$,} \\[12pt]
                          \left\lceil \displaystyle \frac{2k-1}{2k-2}n - \frac{k}{k-1} \right\rceil & \text{if $k \ge 2$,}
                         \end{cases} \]
 and these lower bounds are sharp.
\end{theorem}
\begin{proof}
 First, assume $k=1$. It is not difficult to see that in this case, only the complete graphs are internally color-avoiding connected. Thus the minimum number of edges of an internally color-avoiding connected graph on $n$ vertices, where all the vertices are colored with the same color, is indeed $\binom{n}{2}$.

 \medskip
 
 Now assume $k \ge 2$ and let $G = (V,E)$ be an internally vertex-color-avoiding connected graph on $n$ vertices colored with a color set $C$, where $|C| = k \ge 2$ and each color in $C$ is used at least once. Let $x$ be the number of edges whose endpoints are of the same color, and let $n_c$ be the number of vertices of color $c$ for any $c \in C$. Now we give two lower bounds on the number of edges of $G$ using the variable $x$, and then we combine them to eliminate $x$.

 Let us prove the first lower bound, namely
 \[ (k-2) \cdot |E| \ge (k-1)n - k - x \text{.} \]
 Since $G$ is internally vertex-color-avoiding connected, the graph $G_{\overline{c}}$ is connected for any $c \in C$, so it has at least $n - n_c - 1$ edges. Summing this up for all the colors, we count at least
 \[ \sum_{c \in C} (n - n_c - 1) = kn - n - k = (k-1)n - k \]
 edges, where the edges whose endpoints have the same color are counted $k-1$ times, and the other edges (i.e.\ those ones whose endpoints are not of the same color) are counted $k-2$ times. This implies our first lower bound.

 Now let us prove the second lower bound, namely
 \[ k \cdot |E| \ge kn - k + x \text{.} \]
 Since $G$ is internally vertex-$c$-avoiding connected for any $c \in C$ and there are $k \ge 2$ colors, removing the edges whose both endpoints are of color $c$ leaves the graph connected, so the remaining graph has at least $n-1$ edges. Summing up these remaining edges for all the colors, we count $k(n-1)$ edges, where the edges whose endpoints have different colors are counted $k$ times, and the other edges (i.e.\ those ones whose endpoints are of the same color) are counted $k-1$ times. This implies our second lower bound.

 Adding these two inequalities together, we obtain
 \[ (2k-2) \cdot |E| \ge (2k-1)n - 2k \text{,} \]
 thus
 \[ |E| \ge \left\lceil \frac{2k-1}{2k-2}n-\frac{k}{k-1} \right\rceil \text{.} \]
 
 To show that this lower bound is tight, we construct an internally vertex-color-avoiding connected graph on $n$ vertices colored with exactly $k \ge 2$ colors that has ${\footnotesize \Big\lceil} \frac{2k-1}{2k-2}n - \frac{k}{k-1} {\footnotesize \Big\rceil}$ edges. Clearly, there exist (unique) integers $m, \ell$ such that $n = (2k-2)m + \ell + 3$ and $0 \le \ell \le 2k-3$.

 First, consider the case when $\ell = 0$. Let $G = (V,E)$ and $c \colon V \to \{ 1, \ldots, k \}$ be defined as follows. Let
 \[ V = \big\{ v_{i,j} \bigm| i \in \{ 1, \ldots, 2k-2 \} \text{ and } j \in \{ 1, \ldots, m \} \big\} \cup \{ v_{1, m+1}, \, v_{1, m+2}, \, v_{2k-2, m+1} \} \text{,} \]
 let
 \begin{multline*}
  E = \big\{ v_{1, j} v_{1, j+1} \bigm| j \in \{ 1, \ldots, m+1 \} \big\} \cup \big\{ v_{2k-2, j} v_{2k-2, j+1} \bigm| j \in \{ 1, \ldots, m \} \big\} \\
  \cup \big\{ v_{i, j} v_{i+1, j} \bigm| i \in \{ 1, \ldots, 2k-3 \} \text{ and } j \in \{ 1, \ldots, m \} \big\} \cup \big\{ v_{1, m+1} \, v_{2k-2, m+1}, \; v_{1, m+2} \, v_{2k-2, m+1} \big\} \text{,}
 \end{multline*}
 and let
 \[ c \colon V \to \{ 1, \ldots, k \} \qquad v \mapsto \begin{cases}
1 & \text{if $v = v_{1, j}$ for some $j \in \{ 1, \ldots, m+2 \}$,} \\
\left\lfloor \frac{i}{2} \right\rfloor + 1 & \text{if $v = v_{i, j}$ for some $i \in \{ 2, 3, \ldots, 2k-3 \}$ and $j \in \{ 1, \ldots, m \}$,} \\
k & \text{if $v = v_{2k-2, j}$ for some $j \in \{ 1, \ldots, m+1 \}$.}
\end{cases} \]

 Now consider the case when $\ell \ne 0$. Let $G = (V,E)$ and $c \colon V \to \{ 1, \ldots, k \}$ be defined as follows. Let
 \begin{multline*}
  V = \big\{ v_{i,j} \bigm| i \in \{ 1, \ldots, 2k-2 \} \text{ and } j \in \{ 1, \ldots, m \} \big\} \cup \big\{ v_{i, m+1} \bigm| i \in \{ 1, \ldots, \ell \} \big\} \\
  \cup \{ v_{1, m+2}, \, v_{2k-2, m+1}, \, v_{2k-2, m+2} \} \text{,}
 \end{multline*}
 let
 \begin{multline*}
  E = \big\{ v_{1, j} v_{1, j+1} \bigm| j \in \{ 1, \ldots, m+1 \} \big\} \cup \big\{ v_{2k-2, j} v_{2k-2, j+1} \bigm| j \in \{ 1, \ldots, m+1 \} \big\} \\
  \cup \big\{ v_{i, j} v_{i+1, j} \bigm| i \in \{ 1, \ldots, 2k-3 \} \text{ and } j \in \{ 1, \ldots, m \} \big\} \cup \big\{ v_{i, m+1} v_{i+1, m+1} \bigm| i \in \{ 1, \ldots, \ell-1 \} \big\} \\
  \cup \big\{ v_{\ell, m+1} \, v_{2k-2, m+1}, \; v_{1, m+2} \, v_{2k-2, m+2} \big\} \text{,}
 \end{multline*}
 and let
 \[ c \colon V \to \{ 1, \ldots, k \} \qquad v \mapsto \begin{cases}
1 & \text{if $v = v_{1, j}$ for some $j \in \{ 1, \ldots, m+2 \}$,} \\
\left\lfloor \frac{i}{2} \right\rfloor + 1 & \text{if $v = v_{i, j}$ for some $i \in \{ 2, 3, \ldots, 2k-3 \}$ and $j \in \{ 1, \ldots, m+1 \}$,} \\
k & \text{if $v = v_{2k-2, j}$ for some $j \in \{ 1, \ldots, m+2 \}$.}
\end{cases} \]
 For some examples, see Figure~\ref{fig:vert-constr}. It is not difficult to see that these graphs are internally vertex-color-avoiding connected and have ${\footnotesize \Big\lceil} \frac{2k-1}{2k-2}n - \frac{k}{k-1} {\footnotesize \Big\rceil}$ edges.
\end{proof}

\begin{figure}[H]
\centering
\begin{tikzpicture}[scale=0.78]
 \tikzstyle{vertex_red}=[draw, shape=regular polygon, regular polygon sides=3, fill=red, minimum size=12pt, inner sep=0]
 \tikzstyle{vertex_blue}=[draw, shape=rectangle, fill=blue, minimum size=8pt, inner sep=0]
 \tikzstyle{vertex_green}=[draw, shape=diamond, fill=green, minimum size=11pt, inner sep=0]
 \tikzstyle{vertex_yellow}=[draw, shape=star, fill=yellow, minimum size=11pt, inner sep=0]

 \footnotesize
 
 \node[vertex_red] (v11) at (0,5) [label={[xshift=0pt, yshift=-1pt] $v_{1,1}$}] {};
 \node[vertex_red] (v12) at (1,5) [label={[xshift=0pt, yshift=-1pt] $v_{1,2}$}] {};
 \node[vertex_red] (v13) at (2,5) [label={[xshift=0pt, yshift=-1pt] $v_{1,3}$}] {};
 \node[vertex_blue] (v21) at (0,4) [label={[xshift=-14pt, yshift=-10pt] $v_{2,1}$}] {};
 \node[vertex_blue] (v31) at (0,3) [label={[xshift=-14pt, yshift=-10pt] $v_{3,1}$}] {};
 \node[vertex_green] (v41) at (0,2) [label={[xshift=-14pt, yshift=-10pt] $v_{4,1}$}] {};
 \node[vertex_green] (v51) at (0,1) [label={[xshift=-14pt, yshift=-10pt] $v_{5,1}$}] {};
 \node[vertex_yellow] (v61) at (0,0) [label=below:$v_{6,1}$] {};
 \node[vertex_yellow] (v62) at (1,0) [label=below:$v_{6,2}$] {};

 \draw[very thick] (v11) -- (v12) -- (v13);
 \draw[very thick] (v61) -- (v62);
 \draw[very thick] (v11) -- (v21) -- (v31) -- (v41) -- (v51) -- (v61);
 \draw[very thick] (v12) -- (v62);
 \draw[very thick] (v13) -- (v62);

 \begin{scope}[shift={(3.5,0)}]
 \node[vertex_red] (v11) at (0,5) [label={[xshift=0pt, yshift=-1pt] $v_{1,1}$}] {};
 \node[vertex_red] (v12) at (1,5) [label={[xshift=0pt, yshift=-1pt] $v_{1,2}$}] {};
 \node[vertex_red] (v13) at (2,5) [label={[xshift=0pt, yshift=-1pt] $v_{1,3}$}] {};
 \node[vertex_blue] (v21) at (0,4) [label={[xshift=-14pt, yshift=-10pt] $v_{2,1}$}] {};
 \node[vertex_blue] (v31) at (0,3) [label={[xshift=-14pt, yshift=-10pt] $v_{3,1}$}] {};
 \node[vertex_green] (v41) at (0,2) [label={[xshift=-14pt, yshift=-10pt] $v_{4,1}$}] {};
 \node[vertex_green] (v51) at (0,1) [label={[xshift=-14pt, yshift=-10pt] $v_{5,1}$}] {};
 \node[vertex_yellow] (v61) at (0,0) [label=below:$v_{6,1}$] {};
 \node[vertex_yellow] (v62) at (1,0) [label=below:$v_{6,2}$] {};
 \node[vertex_yellow] (v63) at (2,0) [label=below:$v_{6,3}$] {};

 \draw[very thick] (v11) -- (v12) -- (v13);
 \draw[very thick] (v61) -- (v62) -- (v63);
 \draw[very thick] (v11) -- (v21) -- (v31) -- (v41) -- (v51) -- (v61);
 \draw[very thick] (v12) -- (v62);
 \draw[very thick] (v13) -- (v63);
 \end{scope}

 \begin{scope}[shift={(7,0)}]
 \node[vertex_red] (v11) at (0,5) [label={[xshift=0pt, yshift=-1pt] $v_{1,1}$}] {};
 \node[vertex_red] (v12) at (1,5) [label={[xshift=0pt, yshift=-1pt] $v_{1,2}$}] {};
 \node[vertex_red] (v13) at (2,5) [label={[xshift=0pt, yshift=-1pt] $v_{1,3}$}] {};
 \node[vertex_blue] (v21) at (0,4) [label={[xshift=-14pt, yshift=-10pt] $v_{2,1}$}] {};
 \node[vertex_blue] (v22) at (1,4) [label={[xshift=13pt, yshift=-10pt] $v_{2,2}$}] {};
 \node[vertex_blue] (v31) at (0,3) [label={[xshift=-14pt, yshift=-10pt] $v_{3,1}$}] {};
 \node[vertex_green] (v41) at (0,2) [label={[xshift=-14pt, yshift=-10pt] $v_{4,1}$}] {};
 \node[vertex_green] (v51) at (0,1) [label={[xshift=-14pt, yshift=-10pt] $v_{5,1}$}] {};
 \node[vertex_yellow] (v61) at (0,0) [label=below:$v_{6,1}$] {};
 \node[vertex_yellow] (v62) at (1,0) [label=below:$v_{6,2}$] {};
 \node[vertex_yellow] (v63) at (2,0) [label=below:$v_{6,3}$] {};
 
 \draw[very thick] (v11) -- (v12) -- (v13);
 \draw[very thick] (v61) -- (v62) -- (v63);
 \draw[very thick] (v11) -- (v21) -- (v31) -- (v41) -- (v51) -- (v61);
 \draw[very thick] (v12) -- (v22) -- (v62);
 \draw[very thick] (v13) -- (v63);
 \end{scope}
 
 \begin{scope}[shift={(10.5,0)}]
 \node[vertex_red] (v11) at (0,5) [label={[xshift=0pt, yshift=-1pt] $v_{1,1}$}] {};
 \node[vertex_red] (v12) at (1,5) [label={[xshift=0pt, yshift=-1pt] $v_{1,2}$}] {};
 \node[vertex_red] (v13) at (2,5) [label={[xshift=0pt, yshift=-1pt] $v_{1,3}$}] {};
 \node[vertex_blue] (v21) at (0,4) [label={[xshift=-14pt, yshift=-10pt] $v_{2,1}$}] {};
 \node[vertex_blue] (v22) at (1,4) [label={[xshift=13pt, yshift=-10pt] $v_{2,2}$}] {};
 \node[vertex_blue] (v31) at (0,3) [label={[xshift=-14pt, yshift=-10pt] $v_{3,1}$}] {};
 \node[vertex_blue] (v32) at (1,3) [label={[xshift=13pt, yshift=-10pt] $v_{3,2}$}] {};
 \node[vertex_green] (v41) at (0,2) [label={[xshift=-14pt, yshift=-10pt] $v_{4,1}$}] {};
 \node[vertex_green] (v51) at (0,1) [label={[xshift=-14pt, yshift=-10pt] $v_{5,1}$}] {};
 \node[vertex_yellow] (v61) at (0,0) [label=below:$v_{6,1}$] {};
 \node[vertex_yellow] (v62) at (1,0) [label=below:$v_{6,2}$] {};
 \node[vertex_yellow] (v63) at (2,0) [label=below:$v_{6,3}$] {};
 
 \draw[very thick] (v11) -- (v12) -- (v13);
 \draw[very thick] (v61) -- (v62) -- (v63);
 \draw[very thick] (v11) -- (v21) -- (v31) -- (v41) -- (v51) -- (v61);
 \draw[very thick] (v12) -- (v22) -- (v32) -- (v62);
 \draw[very thick] (v13) -- (v63);
 \end{scope}
 
 \begin{scope}[shift={(14,0)}]
 \node[vertex_red] (v11) at (0,5) [label={[xshift=0pt, yshift=-1pt] $v_{1,1}$}] {};
 \node[vertex_red] (v12) at (1,5) [label={[xshift=0pt, yshift=-1pt] $v_{1,2}$}] {};
 \node[vertex_red] (v13) at (2,5) [label={[xshift=0pt, yshift=-1pt] $v_{1,3}$}] {};
 \node[vertex_blue] (v21) at (0,4) [label={[xshift=-14pt, yshift=-10pt] $v_{2,1}$}] {};
 \node[vertex_blue] (v22) at (1,4) [label={[xshift=13pt, yshift=-10pt] $v_{2,2}$}] {};
 \node[vertex_blue] (v31) at (0,3) [label={[xshift=-14pt, yshift=-10pt] $v_{3,1}$}] {};
 \node[vertex_blue] (v32) at (1,3) [label={[xshift=13pt, yshift=-10pt] $v_{3,2}$}] {};
 \node[vertex_green] (v41) at (0,2) [label={[xshift=-14pt, yshift=-10pt] $v_{4,1}$}] {};
 \node[vertex_green] (v42) at (1,2) [label={[xshift=13pt, yshift=-10pt] $v_{4,2}$}] {};
 \node[vertex_green] (v51) at (0,1) [label={[xshift=-14pt, yshift=-10pt] $v_{5,1}$}] {};
 \node[vertex_yellow] (v61) at (0,0) [label=below:$v_{6,1}$] {};
 \node[vertex_yellow] (v62) at (1,0) [label=below:$v_{6,2}$] {};
 \node[vertex_yellow] (v63) at (2,0) [label=below:$v_{6,3}$] {};
 
 \draw[very thick] (v11) -- (v12) -- (v13);
 \draw[very thick] (v61) -- (v62) -- (v63);
 \draw[very thick] (v11) -- (v21) -- (v31) -- (v41) -- (v51) -- (v61);
 \draw[very thick] (v12) -- (v22) -- (v32) -- (v42) -- (v62);
 \draw[very thick] (v13) -- (v63);
 \end{scope}
 
 \begin{scope}[shift={(17.5,0)}]
 \node[vertex_red] (v11) at (0,5) [label={[xshift=0pt, yshift=-1pt] $v_{1,1}$}] {};
 \node[vertex_red] (v12) at (1,5) [label={[xshift=0pt, yshift=-1pt] $v_{1,2}$}] {};
 \node[vertex_red] (v13) at (2,5) [label={[xshift=0pt, yshift=-1pt] $v_{1,3}$}] {};
 \node[vertex_blue] (v21) at (0,4) [label={[xshift=-14pt, yshift=-10pt] $v_{2,1}$}] {};
 \node[vertex_blue] (v22) at (1,4) [label={[xshift=13pt, yshift=-10pt] $v_{2,2}$}] {};
 \node[vertex_blue] (v31) at (0,3) [label={[xshift=-14pt, yshift=-10pt] $v_{3,1}$}] {};
 \node[vertex_blue] (v32) at (1,3) [label={[xshift=13pt, yshift=-10pt] $v_{3,2}$}] {};
 \node[vertex_green] (v41) at (0,2) [label={[xshift=-14pt, yshift=-10pt] $v_{4,1}$}] {};
 \node[vertex_green] (v42) at (1,2) [label={[xshift=13pt, yshift=-10pt] $v_{4,2}$}] {};
 \node[vertex_green] (v51) at (0,1) [label={[xshift=-14pt, yshift=-10pt] $v_{5,1}$}] {};
 \node[vertex_green] (v52) at (1,1) [label={[xshift=13pt, yshift=-10pt] $v_{5,2}$}] {};
 \node[vertex_yellow] (v61) at (0,0) [label=below:$v_{6,1}$] {};
 \node[vertex_yellow] (v62) at (1,0) [label=below:$v_{6,2}$] {};
 \node[vertex_yellow] (v63) at (2,0) [label=below:$v_{6,3}$] {};
 
 \draw[very thick] (v11) -- (v12) -- (v13);
 \draw[very thick] (v61) -- (v62) -- (v63);
 \draw[very thick] (v11) -- (v21) -- (v31) -- (v41) -- (v51) -- (v61);
 \draw[very thick] (v12) -- (v22) -- (v32) -- (v42) -- (v52) -- (v62);
 \draw[very thick] (v13) -- (v63);
 \end{scope}
\end{tikzpicture}
\caption{Some internally vertex-color-avoiding connected graphs on $n \in \{ 9, 10, \ldots, 14 \}$ vertices colored with exactly $k = 4$ colors and with ${\scriptsize \Big\lceil} \frac{2k-1}{2k-2}n - \frac{k}{k-1} {\scriptsize \Big\rceil}$ of edges. (In all these cases, the value of $m$ is 1, and the values of $\ell$ are $0, 1, \ldots, 5$, respectively.)}
\label{fig:vert-constr}
\end{figure}
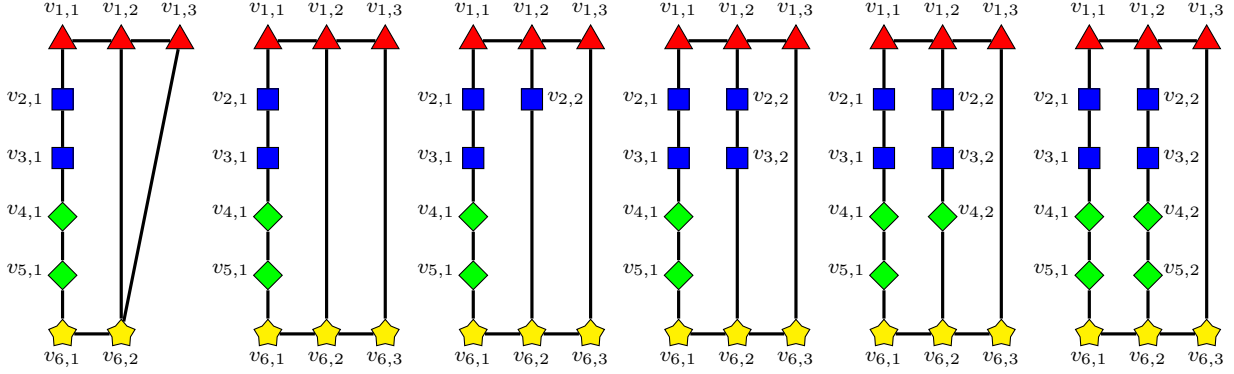

Now we present a polynomial-time approximation algorithm for finding an internally vertex-color-avoiding connected spanning subgraph with minimum number of edges. To simplify the description of the algorithm, we introduce the following subroutines. The subroutine \textbf{DifferentColorNeighbor}$(v, G)$ checks whether the vertex $v$ has a neighbor whose color is different from the color of $v$ in the vertex-colored graph $G$, and the subroutine \textbf{DifferentColorEdge}$(v, G)$ returns an edge which connects $v$ to such a neighbor.

The algorithm is presented as Algorithm~\ref{alg:ivertex}, and it follows a similar scheme as Algorithms~\ref{alg:edge} and~\ref{alg:vertex} --- the only difference is in the first step: in order to ensure internally vertex-color-avoiding connectivity, the initial spanning tree is chosen more carefully. In more detail, Algorithm~\ref{alg:ivertex} does the following. Given an internally vertex-color-avoiding connected graph $G$ colored with a color set $C$, first we pick such a spanning tree of $G$ in which every vertex $v$ has a neighbor whose color is different from that of $v$; it can be shown that such a spanning tree always exists in $G$. Then for any color $c \in C$, if the removal of the vertices of color $c$ disconnects the graph of the so far selected edges, then we contract the components of the graph whose vertex set consists of the vertices which have colors different from $c$ and whose edge set consists of the selected edges. Then we select some edges in the original graph which correspond to the edges of an arbitrary spanning tree in this contracted graph. At the end, the selected edges clearly form an internally vertex-color-avoiding connected spanning subgraph.



\begin{algorithm}[H]
\caption{Finding internally vertex-color-avoiding connected spanning subgraphs} \label{alg:ivertex}
\begin{algorithmic}[1]
\Require An internally vertex-color-avoiding connected graph $G = (V, E)$ colored with a color set $C$.
\Ensure An internally vertex-color-avoiding connected spanning subgraph $G'$ of $G$.

\If{$|C|=1$}
    \State \Return $G$
\Else
\State $E' \gets \emptyset$ \hfill \texttt{\textcolor{black!50}{//~Phase~1}}
\State $G' \gets \text{Graph$(V, E')$}$
\For{$v\in V$}
     \If{not DifferentColorNeighbor$(v, G')$}
         \State $E' \gets E' \cup \big\{ \text{DifferentColorEdge}(v, G) \big\}$
         \State $G' \gets \text{Graph$(V, E')$}$
     \EndIf
\EndFor
\If{$G'$ is not connected} \hfill \texttt{\textcolor{black!50}{//~Phase~2}}
     \State $\mathcal{W} \gets \text{ConnectedComponents}(G')$
     \State $H \gets \text{ContractVertices}(G,\mathcal{W})$
     \State $E' \gets E' \cup \textrm{BeforeContraction} \big( G, \; H, \; \text{SpanningTree}(H) \big)$
     \State $G' \gets \text{Graph$(V, E')$}$
\EndIf
\For{$c\in C$} \hfill \texttt{\textcolor{black!50}{//~Phase~3}}
     \If{$G'_{\overline{c}}$ is not connected}
         \State $\mathcal{W} \gets \text{ConnectedComponents}(G'_{\overline{c}})$
         \State $H \gets \text{ContractVertices}(G_{\overline{c}},\mathcal{W})$
         \State $E' \gets E' \cup \textrm{BeforeContraction} \big( G, \; H, \; \text{SpanningTree}(H) \big)$
         \State $G' \gets \text{Graph$(V, E')$}$
     \EndIf
\EndFor
\EndIf

\State \Return $G'$
\end{algorithmic}
\end{algorithm}

Again, let us note that similarly to Algorithms~\ref{alg:matr}--\ref{alg:vertex}, at the end of Algorithm~\ref{alg:ivertex} we can greedily deselect those edges whose removal from the graph of the so far selected edges does not destroy the internally vertex-color-avoiding property of this graph.

Now we are ready to analyze Algorithm~\ref{alg:ivertex}.

\begin{theorem} \label{thm:ivertexapprox}
 Algorithm~\ref{alg:ivertex} is a polynomial-time $\big( 2 \cdot \frac{2k-2}{2k-1} \big)$-approximation algorithm for the problem of finding an internally vertex-color-avoiding connected spanning subgraph with minimum number of edges in an internally vertex-color-avoiding connected graph whose vertices are colored with exactly $k \in \mathbb{Z}_+$ colors.
\end{theorem}
\begin{proof}
 We discuss the time complexity and the correctness of the algorithm separately.
 
 \medskip
 
 \textit{Correctness.} If $k=1$, then by Theorem~\ref{thm:ivertex}, the input of Algorithm~\ref{alg:ivertex} must be a complete graph, and so does the output, thus in this case the algorithm finds an optimal solution.

 \medskip
 
 Now assume $k \ge 2$. First, we show that at the end of Phase~2, the so far selected edges form such a spanning tree of $G$, in which every vertex $v$ has a neighbor whose color is not $c(v)$, where $c(v)$ denotes the color of $v$. To show this, we begin with proving that the edges selected in Phase~1 cannot induce a cycle. Suppose to the contrary that there exists a cycle $v_1 v_2 \ldots v_{\ell}$ for some integer $\ell \ge 3$ such that all the edges $v_1 v_2$, $v_2 v_3$, $\ldots$, $v_{\ell - 1} v_{\ell}$, and $v_{\ell} v_1$ were selected in Phase~1. Without loss of generality, we can assume that among these edges, $v_1 v_2$ was selected last. By the choice of the edges in Phase~1, we know that the color of $v_1$ is different from that of $v_{\ell}$, and the color of $v_2$ is different from that of $v_3$. (Note that if $\ell = 3$, then $v_{\ell} = v_3$.) Since the edge $v_1 v_2$ is the last selected edge of the cycle, this means that at the time of its selection, both $v_1$ and $v_2$ had such neighbors in the graph of the so far selected edges, whose colors are different from theirs. Thus the edge $v_1 v_2$ should not have been selected, which is a contradiction. Therefore, the edges selected in Phase~1 form a forest, and clearly, in Phase~2 this forest is extended to a spanning tree of $G$. Also note that by the internally vertex-color-avoiding connectivity of $G$, every vertex $v$ has a neighbor whose color is different from that of $v$.
 
 By the same reasoning as that for Algorithm~\ref{alg:vertex}, after Phase~3, the graph of the so far selected edges is vertex-color-avoiding connected. To prove that it is internally vertex-color-avoiding connected, we need to show that any vertex $v$ is internally vertex-$c(v)$-avoiding connected to any other vertex $w$. So let $v$ and $w$ be two arbitrary vertices.

 If $w$ is not of color $c(v)$, then in Phase~1, the algorithm selected an edge $uv$ where the color of $u$ is different from $c(v)$. Note that such an edge exists since the input graph $G$ is internally vertex-color-avoiding connected. Since $u$ and $w$ are vertex-$c(v)$-avoiding connected in the graph of the selected edges when the algorithm terminates and neither $u$ nor $w$ are of color $c(v)$, there exists a vertex-$c(v)$-avoiding path between $u$ and $w$ in the graph of the selected edges. Since $v$ is of color $c(v)$, this path cannot contain $v$, so extending this path with the edge $uv$, we obtain an internally vertex-$c(v)$-avoiding path between $v$ and $w$ in the graph of the selected edges at the end of Phase~3.
 
 If $w$ is of color $c(v)$, then in Phase~1, the algorithm selected two edges $u_1 v$ and $u_2 w$ where neither $u_1$ nor $u_2$ is of color $c(v)$. Note that such edges exist since the input graph $G$ is internally vertex-color-avoiding connected, and also note that $u_1$ and $u_2$ are not necessarily distinct vertices. If $u_1 = u_2$, then $v u_1 w$ is an internally vertex-$c(v)$-avoiding path between $v$ and $w$ in the graph of the selected edges when the algorithm terminates. If $u_1 \ne u_2$, then since neither $u_1$ nor $u_2$ is of color $c(v)$, there exists a vertex-$c(v)$-avoiding path between $u_1$ and $u_2$ in the graph of the selected edges at the end of Phase~3. Extending this path with the edges $u_1 v$ and $u_2 w$, we obtain an internally vertex-$c(v)$-avoiding path between $v$ and $w$ in the graph of the selected edges when the algorithm terminates.

 Therefore, the algorithm finds an internally vertex-color-avoiding connected spanning subgraph.
 
 \medskip
 
 \textit{Approximation factor.} Now we prove that the output has at most 
 \[ 2n-3 \le 2 \cdot \frac{2k-2}{2k-1} \cdot \left\lceil \frac{2k-1}{2k-2}n - \frac{k}{k-1} \right\rceil \]
 edges, which is, by Theorem~\ref{thm:ivertex}, at most $2 \cdot \frac{2k-2}{2k-1}$ times as many as the minimum number of edges of an internally vertex-color-avoiding connected graph whose vertices are colored with exactly $k$ colors, implying that Algorithm~\ref{alg:ivertex} is a $\big( 2 \cdot \frac{2k-2}{2k-1} \big)$-approximation algorithm.

 As we showed before, at the end of Phase~2 the so far selected edges form a spanning tree of $G$. Thus, at the end of Phase~2, $n-1$ edges are selected. By the same reasoning as that for Algorithm~\ref{alg:vertex}, at the end of Phase~3, at most $2n-3$ edges are selected. Thus the output has indeed at most $2n-3$ edges.

 \medskip
 
 \textit{Time complexity.} Similarly to Algorithm~\ref{alg:vertex}, Algorithm~\ref{alg:ivertex} also runs in $O \big( |C| \cdot |V|^2 \big)$ time.
\end{proof}

\medskip

As we observed in the proof of Theorem~\ref{thm:ivertexapprox}, if $k=1$ then Algorithm~\ref{alg:ivertex} always finds an optimal solution. Similarly to Algorithm~\ref{alg:vertex}, now we show that the approximation factor of Algorithm~\ref{alg:ivertex} cannot be improved by a more precise analysis. More precisely, for $k \ge 2$, we present some internally vertex-color-avoiding connected graphs $G$, for which the output of the algorithm can have exactly $2n - 3$ edges, while the optimal solution has ${\footnotesize \Big\lceil} \frac{2k-1}{2k-2}n - \frac{k}{k-1} {\footnotesize \Big\rceil}$ edges. For this, let us consider the graph $G' = (V,E')$ constructed in the proof of Theorem~\ref{thm:ivertex} on $n$ vertices, where $(2k-2) \mid (n - 3)$ and $n \ge 4k - 1$ (with the notation used to describe the construction, this means $m \ge 2$ and $\ell = 0$).
 
First, let us find a permutation $w_0, w_1, \ldots, w_{n-1}$ of the vertices, in which any three consecutive vertices are all of different colors. Let
\[ w_{2j(k-1)} \colonequals v_{1,j+1} \]
and
\[ w_{i+2j(k-1)} \colonequals v_{2i,j+1} \]
and
\[ w_{(2j+1)(k-1)} \colonequals v_{2k-2,j+1} \]
for any $i \in \{ 1, \ldots, k-2 \}$ and $j \in \{ 0, 1, \ldots, m-1 \}$. Let
\[ w_{i+(2j+1)(k-1)} \colonequals v_{2i+1,j+1} \]
for any $i \in \{ 1, \ldots, k-2 \}$ and $j \in \{ 0, 1, \ldots, m-2 \}$. Let
\[ w_{(2m-1)(k-1)+1} \colonequals v_{1,m+2} \text{.} \]
Let
\[ w_{i+(2m-1)(k-1)+1} \colonequals v_{2i+1,m} \]
for any $i \in \{ 1, \ldots, k-2 \}$. Finally, let
\[ w_{2m(k-1)+1} \colonequals v_{1,m+1} \]
and
\[ w_{2m(k-1)+2} \colonequals v_{2k-2,m+1} \text{.} \]

Let $G = (V, E' \cup E'')$, where
\[ E'' \colonequals \big\{ w_j w_{j+1} \bigm| j \in \{ 0, 1, \ldots, n-2 \} \big\} \cup \big\{ w_j w_{j+2} \bigm| j \in \{ 0, 1, \ldots, n-3 \} \big\} \text{.} \]
For an example, see Figure~\ref{fig:ivertex-algo}.

By Theorem~\ref{thm:ivertex}, the subgraph $\big( V, E')$ is an optimal solution. However, the output of Algorithm~\ref{alg:ivertex} can also be the subgraph $(V, E'')$: the algorithm can select the edges of $\big\{ w_j w_{j+1} \bigm| j \in \{ 0, 1, \ldots, n-2 \} \big\}$ in Phase~1 (in the order $w_1 w_0$, $w_2 w_1$, $w_3 w_2$, \ldots, $w_{n-1} w_{n-2}$), then since these edges obviously form a Hamiltonian path and thus a spanning tree of $G$, no further edges get selected in Phase~2, and finally, the algorithm selects all the edges of $\big\{ v_j v_{j+2} \bigm| j \in \{ 0, 1, \ldots, n-3 \}$ in Phase~3. Clearly, this output has $2n-3$ edges.

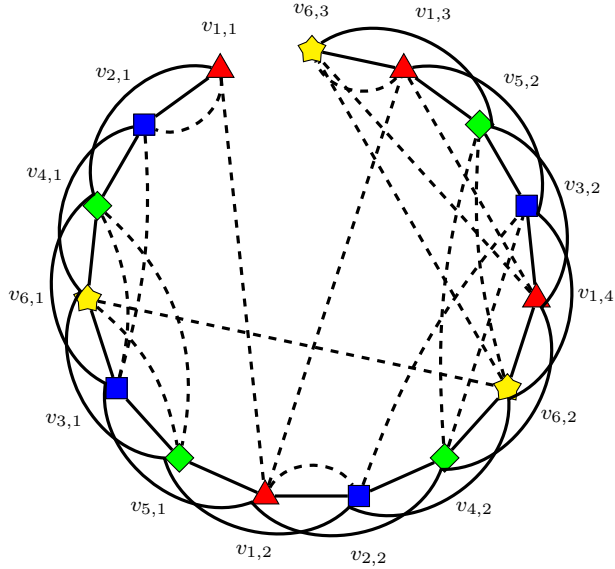
\begin{figure}[H]
\centering
\begin{tikzpicture}[scale=3]
 \tikzstyle{vertex_red}=[draw, shape=regular polygon, regular polygon sides=3, fill=red, minimum size=12pt, inner sep=0]
 \tikzstyle{vertex_blue}=[draw, shape=rectangle, fill=blue, minimum size=8pt, inner sep=0]
 \tikzstyle{vertex_green}=[draw, shape=diamond, fill=green, minimum size=11pt, inner sep=0]
 \tikzstyle{vertex_yellow}=[draw, shape=star, fill=yellow, minimum size=11pt, inner sep=0]

 \footnotesize
 
 \node[vertex_red] (a0) at (90+360/15+0*360/15:1) [label={[shift={($(0,-0.4)+(90:0.5)$)}] $v_{1,1}$}] {};
 \node[vertex_blue] (a1) at (90+360/15+1*360/15:1) [label={[shift={($(0,-0.4)+(120:0.8)$)}] $v_{2,1}$}] {};
 \node[vertex_green] (a2) at (90+360/15+2*360/15:1) [label={[shift={($(0,-0.4)+(150:0.8)$)}] $v_{4,1}$}] {};
 \node[vertex_yellow] (a3) at (90+360/15+3*360/15:1) [label={[shift={($(0,-0.4)+(180:0.8)$)}] $v_{6,1}$}] {};
 \node[vertex_blue] (a4) at (90+360/15+4*360/15:1) [label={[shift={($(0,-0.4)+(210:0.8)$)}] $v_{3,1}$}] {};
 \node[vertex_green] (a5) at (90+360/15+5*360/15:1) [label={[shift={($(0,-0.4)+(240:0.8)$)}] $v_{5,1}$}] {};
 \node[vertex_red] (a6) at (90+360/15+6*360/15:1) [label={[shift={($(0,-0.4)+(260:0.8)$)}] $v_{1,2}$}] {};
 \node[vertex_blue] (a7) at (90+360/15+7*360/15:1) [label={[shift={($(0,-0.4)+(280:0.8)$)}] $v_{2,2}$}] {};
 \node[vertex_green] (a8) at (90+360/15+8*360/15:1) [label={[shift={($(0,-0.4)+(300:0.8)$)}] $v_{4,2}$}] {};
 \node[vertex_yellow] (a9) at (90+360/15+9*360/15:1) [label={[shift={($(0,-0.4)+(330:0.8)$)}] $v_{6,2}$}] {};
 \node[vertex_red] (a10) at (90+360/15+10*360/15:1) [label={[shift={($(0,-0.4)+(0:0.8)$)}] $v_{1,4}$}] {};
 \node[vertex_blue] (a11) at (90+360/15+11*360/15:1) [label={[shift={($(0,-0.4)+(20:0.8)$)}] $v_{3,2}$}] {};
 \node[vertex_green] (a12) at (90+360/15+12*360/15:1) [label={[shift={($(0,-0.4)+(45:0.8)$)}] $v_{5,2}$}] {};
 \node[vertex_red] (a13) at (90+360/15+13*360/15:1) [label={[shift={($(0,-0.4)+(60:0.8)$)}] $v_{1,3}$}] {};
 \node[vertex_yellow] (a14) at (90+360/15+14*360/15:1) [label={[shift={($(0,-0.4)+(90:0.5)$)}] $v_{6,3}$}] {};
 \draw[very thick] (a0) -- (a1) -- (a2) -- (a3) -- (a4) -- (a5) -- (a6) -- (a7) -- (a8) -- (a9) -- (a10) -- (a11) -- (a12) -- (a13) -- (a14);
 \draw[very thick] (a0) to [bend right=60] (a2) to [bend right=60] (a4) to [bend right=60] (a6) to [bend right=60] (a8) to [bend right=60] (a10) to [bend right=60] (a12) to [bend right=60] (a14);
 \draw[very thick] (a1) to [bend right=60] (a3) to [bend right=60] (a5) to [bend right=60] (a7) to [bend right=60] (a9) to [bend right=60] (a11) to [bend right=60] (a13);
 \draw[very thick,dashed] (a0) to [bend left=60] (a1);
 \draw[very thick,dashed] (a0) to [bend left=0] (a6);
 \draw[very thick,dashed] (a1) to [bend left=10] (a4);
 \draw[very thick,dashed] (a2) to [bend left=20] (a4);
 \draw[very thick,dashed] (a2) to [bend left=30] (a5);
 \draw[very thick,dashed] (a3) to [bend left=20] (a5);
 \draw[very thick,dashed] (a3) to [bend left=0] (a9);
 \draw[very thick,dashed] (a6) to [bend left=60] (a7);
 \draw[very thick,dashed] (a6) to [bend left=0] (a13);
 \draw[very thick,dashed] (a7) to [bend left=10] (a11);
 \draw[very thick,dashed] (a8) to [bend left=0] (a11);
 \draw[very thick,dashed] (a8) to [bend left=10] (a12);
 \draw[very thick,dashed] (a9) to [bend left=10] (a12);
 \draw[very thick,dashed] (a9) to [bend left=0] (a14);
 \draw[very thick,dashed] (a10) to [bend left=0] (a13);
 \draw[very thick,dashed] (a10) to [bend left=0] (a14);
 \draw[very thick,dashed] (a13) to [bend left=60] (a14);
\end{tikzpicture}
\caption{An internally vertex-color-avoiding connected graph colored with $k=4$ colors and on $n = 15$ vertices, which contains an internally vertex-color-avoiding connected spanning subgraph with ${\footnotesize \Big\lceil} \frac{2k-1}{2k-2} n - \frac{k}{k-1} {\footnotesize \Big\rceil} = 17$ edges --- for example, the dashed edges form such a subgraph --- and an internally vertex-color-avoiding connected spanning subgraph with $2n-3 = 27$ edges --- that subgraph is spanned by the remaining (i.e.\ solid) edges. The output of Algorithm~\ref{alg:ivertex} can be this latter subgraph.}
\label{fig:ivertex-algo}
\end{figure}

\medskip

Let us consider those internally vertex-color-avoiding connected graphs, which do not contain any internally vertex-color-avoiding connected spanning subgraph except themselves. By the proof of Theorem~\ref{thm:ivertexapprox}, such a graph $G=(V,E)$ on $n$ vertices has at most $2n-3$ edges. In addition, by the proof of Theorem~\ref{thm:ivertexapprox}, if $k = 1$, then such a graph $G$ is a complete graph. Now we show that if $k \in \{ 2, 3 \}$, then the upper bound $|E| \le 2n-3$ is tight. To see this for $k=2$, let us define $G$ as follows.
Let
\[ C = \{ 0, 1 \} \text{,} \]
and let
\[ V = \{ v_0, v_1, \ldots, v_{n-1} \} \]
where $v_0$ is of color 0 and all the other vertices are of color 1. Let
\[ E = \big\{ v_0 v_j \mid j \in \{ 1, 2, \ldots, n-1 \} \big\} \cup \big\{ v_j v_{j+1} \mid j \in \{ 1, 2, \ldots, n-2 \} \big\} \text{.} \]
For an example, see Figure~\ref{fig:ivertex-sharp_k2}. Note that if $k = 3$, then the upper bound $|E| \le 2n-3$ is tight which the same construction shows that was given for the case of vertex-color-avoiding connectivity before Corollary~\ref{corollary:vertex}; for an example, see Figure~\ref{fig:vertex-sharp}.

\begin{figure}[H]
\centering
\begin{tikzpicture}[scale=1.2]
 \tikzstyle{vertex_red}=[draw, shape=regular polygon, regular polygon sides=3, fill=red, minimum size=12pt, inner sep=0]
 \tikzstyle{vertex_blue}=[draw, shape=rectangle, fill=blue, minimum size=8pt, inner sep=0]
 \tikzstyle{vertex_green}=[draw, shape=diamond, fill=green, minimum size=11pt, inner sep=0]
 \tikzstyle{vertex_yellow}=[draw, shape=star, fill=yellow, minimum size=11pt, inner sep=0]
 
 \node[vertex_red] (a0) at (0,0) [label={above left:$v_0$}] {};
 \node[vertex_blue] (a1) at (90+1*72:1) [label={left:$v_1$}] {};
 \node[vertex_blue] (a2) at (90+2*72:1) [label={below:$v_2$}] {};
 \node[vertex_blue] (a3) at (90+3*72:1) [label={below:$v_3$}] {};
 \node[vertex_blue] (a4) at (90+4*72:1) [label={right:$v_4$}] {};
 \node[vertex_blue] (a5) at (90+5*72:1) [label={above:$v_5$}] {};
 \draw[very thick] (a1) -- (a2) -- (a3) -- (a4) -- (a5);
 \draw[very thick] (a0) -- (a1);
 \draw[very thick] (a0) -- (a2);
 \draw[very thick] (a0) -- (a3);
 \draw[very thick] (a0) -- (a4);
 \draw[very thick] (a0) -- (a5);
\end{tikzpicture}
\caption{An internally vertex-color-avoiding connected graph on $n = 6$ vertices colored with exactly $k = 2$ colors having $2n-3 = 9$ edges and having the property that none of the edges can be removed such that the graph remains internally vertex-color-avoiding connected.}
\label{fig:ivertex-sharp_k2}
\end{figure}
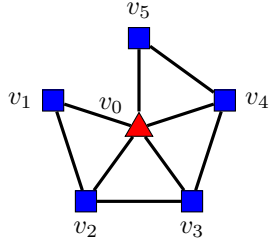

Therefore, we obtained the following.

\begin{corollary}\label{corollary:ivertex}
 Let $G=(V,E)$ be an internally vertex-color-avoiding connected graph on $n$ vertices such that $G-e$ is not internally vertex-color-avoiding connected for any $e\in E$. Then if the vertices are colored with exactly $k=1$ colors then $|E| = \binom{n}{2}$, and if $k \ge 2$ then $|E| \le 2n-3$. Moreover, if $k \in \{ 2,3 \}$, then this upper bound is sharp.
\end{corollary}

By Corollary~\ref{corollary:ivertex}, a simple greedy algorithm (i.e.\ an analogous version of the one described in Remark~\ref{remark:simple_alg}) also performs in polynomial time and with the same approximation factor as Algorithm~\ref{alg:ivertex}.

\section{Conclusions}

In this article, we considered the problems of finding edge-, vertex-, and internally vertex-color-avoiding connected spanning subgraphs with minimum number of edges and finding courteously colored rank preserving restrictions of a matroid to a set of minimum size. We proved that all of these problems are NP-hard and provided polynomial-time approximation algorithms for them. We also gave sharp lower bounds on the number of edges in edge-, vertex- and internally vertex-color-avoiding connected graphs and on the number of elements in courteously colored matroids.

Some possible generalizations of this framework are $\ell$-color-avoiding $k$-edge- and $k$-vertex-connectivity. An edge-colored graph $G$ is called \emph{edge-$\ell$-color-avoiding $k$-vertex-} or \emph{$k$-edge-connected} if after the removal of the edges of any $\ell$ colors from $G$, the remaining graph is $k$-vertex- or $k$-edge-connected, respectively, for some $k, \ell \in \mathbb{Z}_+$. The notion of \emph{vertex-$\ell$-color-avoiding $k$-vertex-} and \emph{$k$-edge-connectivity}, and the notion of \emph{internally vertex-$\ell$-color-avoiding $k$-vertex-} or \emph{$k$-edge-connectivity} can be defined analogously. 

Then we can consider the problems of finding $\ell$-color-avoiding $k$-edge- or $k$-vertex-connected spanning subgraphs with minimum number of edges, or we can also investigate an edge-weighted version of them in which instead of minimizing the number of edges, we want to minimize the sum of the edge-weights of such spanning subgraphs. Clearly, these problems are also NP-hard since they are NP-hard even in the case $k = \ell = 1$ without edge-weights. However, developing approximation and fixed-parameter tractable algorithms for them seems to be an interesting problem that is subject of future research.

\paragraph{Acknowledgement.} The authors would like to express their gratitude to Roland Molontay for useful conversations and for his support during the research process. The authors are also grateful to Krist\'{o}f B\'{e}rczi for his comments regarding the manuscript. The research presented in this paper was supported by the Ministry of Culture and Innovation and the National Research, Development and Innovation Office within the framework of the Artificial Intelligence National Laboratory Programme.
The research of Kitti Varga was supported by the Hungarian National Research, Development and Innovation Office -- NKFIH, grant number FK128673. The reserach of J\'{o}zsef Pint\'{e}r was funded by the New National Excellence Program of the of the National Research, Development and Innovation Fund (\'{U}NKP-23-3).

\end{document}